\documentclass[11pt]{article}
\usepackage{floatrow}
\newfloatcommand{capbtabbox}{table}[][\FBwidth]

\usepackage{amsmath,amsfonts,amssymb,amsthm,graphics,cite,color,caption}
\usepackage{epsf,epsfig}
\usepackage{psfrag}
\usepackage{ctable}
\usepackage{multicol}
\usepackage{threeparttable}
\usepackage{float}
\usepackage{graphicx}
\usepackage{indentfirst}
\usepackage{epstopdf}
\usepackage{ifpdf}
\usepackage{mathrsfs}
\usepackage{subfigure}
\usepackage{overpic}
\usepackage{lineno}
\usepackage{floatrow,subfig}
\usepackage{multirow}
\allowdisplaybreaks[4]

\usepackage[utf8]{inputenc}
\usepackage[english]{babel}
\usepackage{blindtext}

\newtheorem{theorem}{Theorem}[section]

\newtheorem{definition}{Definition}
\newtheorem{lemma}[theorem]{Lemma}
\newtheorem{proposition}[theorem]{Proposition}
\newtheorem{remark}{Remark}
\DeclareMathOperator{\divergenceOperator}{div}
\DeclareMathOperator{\cycl}{Cycl}

\numberwithin{equation}{section}
\numberwithin{figure}{section}

\theoremstyle{remark}

\theoremstyle{plain}

\allowdisplaybreaks

\begin{document}

\title{\bf {Cyclicity of sliding cycles with singularities of regularized piecewise smooth visible-invisible two-folds\thanks{Research of J.H. was partially supported by the National Natural Science Foundation of China (No. 12471162, 12231008).  Research of O.H.P. is supported by Sao Paulo Research Foundation (FAPESP) grants 2021/10198-9 and 2024/00392-0. Research of J.Y. was partially supported by the China Scholarship Council program (No.202306770071).
Email address: hjc@ccnu.edu.cn (J. Huang), renato.huzak@uhasselt.be (R. Huzak), otavio.perez@icmc.usp.br (O.H. Perez), 2662219361@qq.com (J. Yao).}}}

 \author {{\sc  Jicai Huang$^{a}$, Renato Huzak$^{b}$\footnote{Corresponding author}, Otavio Henrique Perez$^{b,c}$, Jinhui Yao$^{a,b}$}\\
{\it\small $^{a}$  School of Mathematics and Statistics, Central China Normal University,} \\
{\it\small  Wuhan, Hubei 430079, P. R. China}\\
{\it\small  $^{b}$ Hasselt University, Campus Diepenbeek, Agoralaan Gebouw D, 3590 Diepenbeek, Belgium}\\
{\it\small  $^{c}$ Universidade de S\~{a}o Paulo (USP),}
{\it\small Instituto de Ci\^{e}ncias Matem\'aticas e de Computa\c{c}\~{a}o (ICMC).}\\
 {\it\small Avenida Trabalhador S\~{a}o Carlense, 400, CEP 13566-590, S\~{a}o Carlos, S\~{a}o Paulo, Brazil}
}

\date{} \maketitle \normalsize
\maketitle{ {\bf Abstract}. In this paper we study the cyclicity of sliding cycles for regularized piecewise smooth visible-invisible two-folds, in the presence of singularities of the Filippov sliding vector field located away from two-folds. We obtain a slow-fast system after cylindrical blow-up and use a well-known connection between the divergence integral along orbits and transition maps for vector fields. Since properties of the divergence integral depend on the location and multiplicity of singularities, we divide the sliding cycles into different classes, which can then produce different types of cyclicity results. As an example, we apply our results to regularized piecewise linear systems.
\vspace{0.3cm}

\noindent {\em Key words:}~ Sliding cycle, cyclicity, piecewise smooth system, difference map, two-folds.

\section{Introduction}\label{section-intro}
In this paper we consider a family of regularized piecewise smooth (PWS) systems
\begin{equation}\label{PWS}
\dot{z}=Z_{\lambda,c}^{+}(z)\phi(y\epsilon^{-1})+Z^{-}_{\lambda,c}(z)(1-\phi(y\epsilon^{-1})), \quad z:=(x,y)\in \mathbb{R}^{2},
\end{equation}
where $Z_{\lambda,c}^{+}=(X_{\lambda,c}^{+},Y_{\lambda,c}^{+})$ and $Z_{\lambda,c}^{-}=(X_{\lambda,c}^{-},Y_{\lambda,c}^{-})$ are $C^{\infty}$-smooth (shortly, smooth) vector fields, $0<\epsilon\ll 1$, $\lambda\in\mathbb{R}$ is a small parameter and $c$ is kept in a small neighborhood of $c_0\in \mathbb{R}^{l}$. Throughout this paper, we adopt the notation $\lambda\sim 0$ and $c\sim c_0$. The role of $\lambda$ will be explained later and $c$ is included here for the sake of generality.  Moreover, $\phi: \mathbb{R}\rightarrow \mathbb{R}$ is a smooth regularization function satisfying
\begin{equation}\label{regular}
\phi'(s)>0 \quad \text{for} \ \text{all} \ s\in\mathbb{R}, \quad \lim\limits_{s\rightarrow \infty}\phi(s)=1, \quad \lim\limits_{s\rightarrow -\infty}\phi(s)=0.
\end{equation}

System \eqref{PWS} has attracted much attention in the past few years
(see e.g. \cite{HK23,HK24,HKR,JKW21a,JKW21b,KH15,K20} and references therein). Before explaining the importance of \eqref{PWS} in more detail, we first recall some basic definitions in the PWS theory \cite{dBCK08,Filippov,GSTJDE}.

If we let $\epsilon\rightarrow 0$, then system \eqref{PWS} can be written as

\begin{equation}
\dot{z}=\left\{
\begin{array}{ll}
	Z_{\lambda,c}^{+}(z), \quad \text{if}  \quad y>0,\\[3mm]
	Z_{\lambda,c}^{-}(z), \quad \text{if} \quad  y<0,
	\end{array}
\right.
	\label{PWS-1}
	\end{equation}
with the switching manifold being $\Sigma=\{z\in\mathbb R^2 | h(z)=0\}$, where $h(x,y) = y$. System \eqref{PWS} is a regularization of the piecewise smooth vector field \eqref{PWS-1}.

The Lie-derivative of the function $h$ with respect to the vector field $Z_{\lambda,c}^\pm$ at $z$ is given by $Z_{\lambda,c}^\pm(h)(z):=\nabla h(z)\cdot Z_{\lambda,c}^\pm(z)$. Straightforward computations lead to $Z_{\lambda,c}^\pm(h)(z)= Y_{\lambda,c}^\pm(z)$. This allows us to split $\Sigma$ into three different subsets, which we will denote as $\Sigma_{cr}(\lambda,c)$, $\Sigma_{sl}(\lambda,c)$ and $\Sigma_{T}(\lambda,c)$, depending on $\lambda$ and $c$.

The subset $\Sigma_{cr}(\lambda,c)\subset \Sigma$ is called {\it crossing set}, and it consists of all points $q=(x,0)$ satisfying
$$Z_{\lambda,c}^+(h)(q)Z_{\lambda,c}^-(h)(q)=Y_{\lambda,c}^+(q)Y_{\lambda,c}^-(q)>0.$$
At each point $q\in\Sigma_{cr}(\lambda,c)$ the orbit of PWS system \eqref{PWS-1} is the concatenation of the orbit of $Z_{\lambda,c}^+$ and the orbit of $Z_{\lambda,c}^-$ through $q$, crossing the switching line $\Sigma$.

The subset $\Sigma_{sl}(\lambda,c)\subset \Sigma$ is called the {\it sliding set}, and it consists of all points $q=(x,0)$ satisfying
$$Z_{\lambda,c}^+(h)(q)Z_{\lambda,c}^-(h)(q)=Y_{\lambda,c}^+(q)Y_{\lambda,c}^-(q)<0.$$
Moreover, we say that $q\in\Sigma_{sl}(\lambda,c)$ is a {\it stable} (resp. {\it unstable}) sliding point if $Y_{\lambda,c}^+(q)<0$ and $Y_{\lambda,c}^-(q)>0$ (resp. $Y_{\lambda,c}^+(q)>0$ and $Y_{\lambda,c}^-(q)<0$). Near stable sliding points (resp. unstable sliding points) the vector fields $Z_{\lambda,c}^\pm$ point toward (resp. away from) the switching line $\Sigma$.

Following \cite{Filippov}, one can define a vector field along the sliding set $\Sigma_{sl}(\lambda,c)$ as
\begin{align}\label{slvf}
    Z_{\lambda,c}^{sl}(x)&=\frac{1}{(Z_{\lambda,c}^+-Z_{\lambda,c}^-)(h)}\left(Z_{\lambda,c}^+(h)Z_{\lambda,c}^--Z_{\lambda,c}^-(h)Z_{\lambda,c}^+\right)(x,0)\nonumber\\
    &=\left(\frac{\det{Z_{\lambda,c}}(x)}{(Y_{\lambda,c}^{+}-Y_{\lambda,c}^{-})(x,0)},0\right),
\end{align}
where
\begin{equation}\label{slvf-1}
\det{Z_{\lambda,c}}(x):=X_{\lambda,c}^{-}(x,0)Y_{\lambda,c}^{+}(x,0)-X_{\lambda,c}^{+}(x,0)Y_{\lambda,c}^{-}(x,0).
\end{equation}

Observe that for each $(x,0)\in \Sigma_{sl}(\lambda,c)$, $Z_{\lambda,c}^{sl}(x)$ is defined as the convex combination of the two vectors $Z_{\lambda,c}^\pm(x,0)$, tangent to $\Sigma_{sl}(\lambda,c)$ at $(x,0)$. We denote by $X_{\lambda,c}^{sl}(x)$ the first component in \eqref{slvf} and call this one-dimensional vector field the {\it (Filippov) sliding vector field}. If $X_{\lambda,c}^{sl}(x)=0$ for some $(x,0)\in \Sigma_{sl}(\lambda,c)$, then we say that $x$ is a singularity (or zero) of $X_{\lambda,c}^{sl}$. From \eqref{slvf-1} it follows that $x$ is a singularity of $X_{\lambda,c}^{sl}$ if, and only if, $Z_{\lambda,c}^+(x,0)$ and $Z_{\lambda,c}^-(x,0)$ are linearly dependent.

The subset $\Sigma_{T}(\lambda,c)\subset \Sigma$ consists of all points $q=(x,0)$ where $Y_{\lambda,c}^{+}(q)=0$ or $Y_{\lambda,c}^{-}(q)=0$. In this case, $q$ is called a tangency point. We say that $q\in\Sigma_{T}(\lambda,c)$ is a fold point from ``above'' if the orbit of $Z_{\lambda,c}^+$ through $q$ has a quadratic tangency with $\Sigma$ at $q$, that is,
$$Z_{\lambda,c}^+(h)(q)=0 \quad \text{and} \quad (Z_{\lambda,c}^+)^2(h)(q)\neq 0,$$
or, equivalently,
$$Y_{\lambda,c}^{+}(q)=0, \quad X_{\lambda,c}^{+}(q)\neq 0 \quad \text{and} \quad \partial_{x}Y_{\lambda,c}^{+}(q)\neq 0.$$

Notice that we can write the second order Lie-derivative as $(Z_{\lambda,c}^+)^2(h)(q)=X_{\lambda,c}^{+}(q)\partial_{x}Y_{\lambda,c}^{+}(q)$ for $Z_{\lambda,c}^+(h)(q)=0$.
Similarly, we can define a fold point from ``below'' in terms of $Z_{\lambda,c}^-$. A fold point $q$ from above (resp. from below) is visible if $(Z_{\lambda,c}^+)^2(h)(q)> 0$ (resp. $(Z_{\lambda,c}^-)^2(h)(q)<0$) and invisible if $(Z_{\lambda,c}^+)^2(h)(q)< 0$ (resp. $(Z_{\lambda,c}^-)^2(h)(q)> 0$).

We say that $q\in\Sigma_{T}(\lambda,c)$ is a {\it two-fold singularity} if $q$ is a fold point from above and from below. A two-fold singularity $q$ can be visible-visible ($q$ is visible from above and from below), visible-invisible ($q$ is visible from above and invisible from below or visible from below and invisible from above) and invisible-invisible ($q$ is invisible from above and from below). See Figure \ref{fig-2fold}. It is possible to further classify two-fold singularities taking into account the direction of the flow of $Z_{\lambda,c}^\pm$ and $X_{\lambda,c}^{sl}$ defined in \eqref{slvf}. For more details, see \cite{BLM18, GSTJDE, KH15, KRG03}.

\begin{figure}[htp]
\begin{center}
\includegraphics[width=5in]{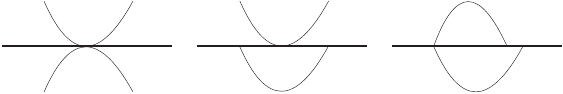}
\caption{\footnotesize{Two-fold singularities. From the left to the right: visible-visible, visible-invisible and invisible-invisible. Each case has sub cases that depend on the direction of the flow of $Z_{\lambda,c}^\pm$ and $X_{\lambda,c}^{sl}$.}}\label{fig-2fold}
\end{center}
\end{figure}

The main purpose of this paper is to study limit cycles of \eqref{PWS}, for $0<\epsilon\ll1$, $\lambda\sim 0$ and $c\sim c_0$, when the PWS system \eqref{PWS-1} satisfies the following conditions for $\lambda=0$:
\begin{enumerate}
    \item[(A1)] The origin $q=(0,0)$ is a visible-invisible two-fold singularity for each $c\sim c_0$ ($q$ is visible from above and invisible from below), with stable sliding for $x < 0$ and unstable sliding for $x > 0$. See Figure \ref{fig:PWS-Smooth}(a).
    \item [(A2)] The sliding vector field $X_{0,c_0}^{sl}$ is positive in a compact segment $[\eta_{-},\eta_{+}]\subset\Sigma$ ($\eta_{-}<0<\eta_{+}$), except at a finite number of singularities of $X_{0,c_0}^{sl}$, located in $[\eta_{-},\eta_{+}]\setminus\{0\}$ and with finite multiplicity. Moreover, we suppose that each singularity located in the interior $(\eta_{-},\eta_{+})$ has even multiplicity.
    \item [(A3)] The (regular) orbit of $Z_{0,c_0}^-$ through the point $(\eta_{+},0)\in \Sigma\cap \{x>0\}$ hits the section $\Sigma\cap \{x<0\}$ in the point $(\eta_{-},0)$ for positive time. See Figure \ref{fig:PWS-Smooth}(a).
    \item [(A4)] The regularization function $\phi$ satisfies \eqref{regular} and $\phi$ is smooth at $\pm\infty$ in the sense that
\begin{equation}\label{smooth}
\phi_{+}(s)=\left\{
\begin{array}{ll}
1  \quad \qquad&\text{for} \quad s=0,\\[3mm]
\phi(s^{-1}) \quad &\text{for} \quad s>0,
\end{array}
\right.\  ;
\quad  \phi_{-}(s)=\left\{
\begin{array}{ll}
\phi(-s^{-1}) \quad \qquad&\text{for} \quad s>0,\\[3mm]
0  \quad &\text{for} \quad s=0,
\end{array}
\right.
\end{equation}
are smooth functions at $s = 0$.
\end{enumerate}

In the rest of this paper we assume that (A1)--(A4) are satisfied. In Section \ref{subsection-assump}, assumptions (A1) and (A2) will be expressed in terms of $X_{0,c}^{\pm}$, $Y_{0,c}^{\pm}$, $\partial_{x}Y_{0,c}^{\pm}$ and $\det{Z_{0,c}}$. In assumption (A2) we require that the sliding vector field can be regularly extended through $x=0$. Assumption (A4) implies that for $z$ kept in a compact set away from the switching manifold $\{y=0\}$, \eqref{PWS} is an $O(\epsilon)$-perturbation of \eqref{PWS-1} where $O(\epsilon)$ is a smooth function in $(z,\epsilon,\lambda,c)$ and goes to zero as $\epsilon\to 0$, uniformly in $(z,\lambda,c)$.

Throughout this paper the points $(\eta_{\pm}, 0)\in\Sigma$ will be called corner points.

We perform the following rescaling in the parameter space $(\epsilon,\lambda)$:
$$(\epsilon,\lambda)=(\tilde \epsilon^2,\tilde\epsilon \tilde{\lambda}),$$
where $0<\tilde \epsilon\ll1$ and $\tilde\lambda\sim 0$. This rescaling will be convenient later when we introduce cylindrical blow-ups and apply \cite{DD08} (see Section \ref{sec3} and Appendix \ref{appendix-extra-blowup}). In particular, $\tilde{\lambda}=0$ implies the existence of a connection $\gamma$ between the attracting and repelling portions of the curve of singularities $\overline C$ in Figure \ref{FIG6} in Appendix \ref{appendix-extra-blowup}. The same scaling has been used in \cite{HK23,HK24}. Now, system \eqref{PWS} can be written as
\begin{equation}\label{PWSRescaled}
\dot{z}=Z_{\epsilon\tilde\lambda,c}^{+}(z)\phi(y\epsilon^{-2})+Z^{-}_{\epsilon\tilde\lambda,c}(z)(1-\phi(y\epsilon^{-2})),
\end{equation}
where we denote $\tilde\epsilon$ again by $\epsilon$ for simplicity of the notation.

Denote by $\Gamma$ the closed curve consisting of the segment $[\eta_{-},\eta_{+}]$ and the orbit of $Z_{0,c_0}^-$ connecting the point $(\eta_{+},0)$ with the point $(\eta_{-},0)$ (see assumptions (A2) and (A3) and Figure \ref{fig:PWS-Smooth}(a)). We call $\Gamma$ a sliding cycle. \textit{The goal of this paper is to provide upper bounds for the number of limit cycles produced by $\Gamma$ inside the family \eqref{PWSRescaled}, with $0<\epsilon\ll1$, $\tilde \lambda\sim 0$ and $c\sim c_0$.}

\begin{figure}[htb]
	\begin{center}		\includegraphics[width=13.8cm,height=3cm]{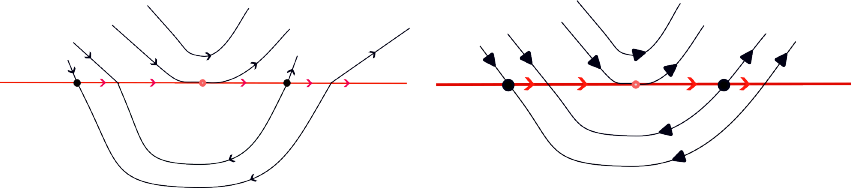}
		{\footnotesize
   \put(-304,-11){\footnotesize{(a)}}
   \put(-365,38){\footnotesize{$\eta_{-}$}}
   \put(-241,38){\footnotesize{$\eta_{+}$}}
   \put(-255,12){\footnotesize{$\Gamma$}}
   \put(-110,-11){\footnotesize{(b)}   \put(33,22){\footnotesize{$\widetilde\Gamma$}
   }}}
         \end{center}
	\caption{\footnotesize{(a) A PWS system with a visible-invisible two-fold singularity at $(x,y)=(0,0)$, and the sliding cycle $\Gamma$. Sliding vector field $X_{0,c_0}^{sl}$ has two singularities in $[\eta_{-},\eta_{+}]$: one corner singularity $x=\eta_-$ and one singularity in the interior $(\eta_{-},\eta_{+})$. (b) Canard cycle $\widetilde\Gamma$ in a smooth slow-fast system. The slow dynamics contains singularities.}}
	\label{fig:PWS-Smooth}
\end{figure}
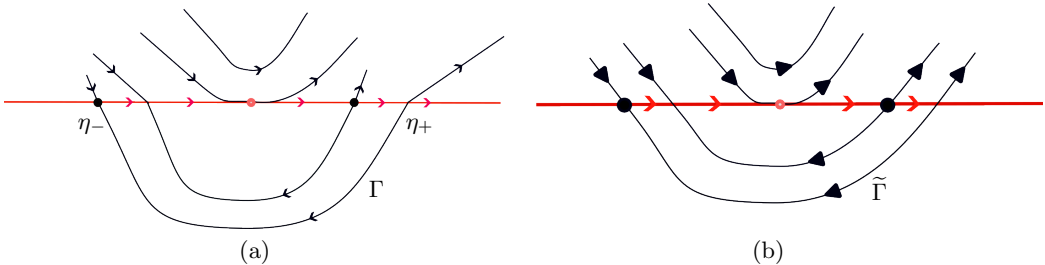

Let us now say more about the interest in systems of the form \eqref{PWSRescaled}. Firstly, since $\phi'>0$, observe that the dynamics of the regularization \eqref{PWSRescaled} of \eqref{PWS-1} near sliding subsets of $\Sigma$ is given by the sliding vector field $X_{0,c}^{sl}$ as $\epsilon\to 0$. More precisely, by using directional blow-up $y=\epsilon^2\widetilde y$, it can be shown that the regularized system \eqref{PWSRescaled} is a slow-fast system whose associated slow dynamics \cite[Chapter 3]{DDR21} is equivalent to the flow of the sliding vector field $X_{0,c}^{sl}$ (see \cite{PanSil}). This fact will play a crucial role throughout our paper. For more details, see Proposition \ref{prop-slow-sliding} in Section \ref{sec3} and \cite[Theorem 2.2]{HK23}.

Furthermore, in \cite{HK23,HK24} it has been observed that the sliding cycle $\Gamma$ as defined above is similar to canard cycles in planar slow-fast systems (see $\widetilde\Gamma$ in Figure \ref{fig:PWS-Smooth}(b)). The canard cycle $\widetilde\Gamma$ contains both normally attracting and normally repelling portions of the line of singularities and a slow–fast Hopf/canard point \cite[Section 2.4]{DDR21} between them. The line of singularities has a quadratic contact with the fast fibers at the slow-fast Hopf point, and the associated slow dynamics goes from the attracting portion to the repelling one with nonzero speed. In addition, the slow dynamics may have singularities located away from the slow-fast Hopf point. It is now clear that our visible-invisible two-fold singularity as described in assumption (A1) will play the role of a canard point in the analysis of limit cycles of \eqref{PWSRescaled}. Notice also that by (A2) the sliding vector field is regularly extended through the origin $x=0$, allowing passage from the stable sliding region to the unstable one with nonzero speed (see \cite{HK23}).

In \cite{HK23}, authors assumed that $X_{0,c_0}^{sl}$ is regular ($X_{0,c_0}^{sl}>0$ on $[\eta_{-},\eta_{+}]$) and studied the existence and number of limit cycles of \eqref{PWSRescaled} produced by $\Gamma$, using the notion of slow divergence integral (see Section \ref{sec2}). The slow divergence integral is important because its zeros provide candidates for limit cycles (see e.g. \cite[Theorem 3.1]{HK23}). We would like to emphasize that the slow divergence integral was first developed by De Maesschalck, Dumortier and Roussarie \cite[Chapter 5]{DDR21}, in the context of slow–fast systems and canard theory. The integral is a key tool in the study of limit cycles Hausdorff close to canard cycles of the type $\widetilde\Gamma$ sketched in Figure \ref{fig:PWS-Smooth}(b) when the slow dynamics is regular. In \cite[Section 1]{HK24}, one can find a planar slow-fast system, with the phase portrait given in Figure \ref{fig:PWS-Smooth}(b), and see how to compute the slow divergence integral for that specific system.

This paper is a natural continuation of \cite{HK23}. In \cite{HK23}, the compact segment $[\eta_{-},\eta_{+}]\subset\Sigma$ does not contain singularities of $X^{sl}_{0,c_{0}}$ away from the two-fold. However, in the present study, the segment $[\eta_{-},\eta_{+}]\subset\Sigma$ may contain singularities of the sliding vector field $X_{0,c_0}^{sl}$ located away from the two-fold singularity, see assumption (A2) and Figure \ref{FIG3}. The main motivation to study this very general model comes from \cite{DD08} where the authors deal with limit cycles produced by canard cycles in the presence of singularities of the slow dynamics (see Figure \ref{fig:PWS-Smooth}(b)). We will prove similar cyclicity results but for regularized PWS systems.  The idea is that we first blow up the switching manifold $\Sigma$ (we use a cylindrical blow-up) and bring system \eqref{PWSRescaled} into the framework of \cite{DD08}, near a curve of singularities located on the top of the cylinder. Then we can study the divergence integral along orbits for the blown-up vector field (near the curve of singularities) by using results of \cite{DD08}. This will be the leading order part in the analysis of the divergence integral along $\Gamma$. For more details we refer to Section \ref{sec3}. We point out that in our setting the slow divergence integral \cite{HK23} is not necessarily well-defined due to the presence of the singularities (see Section \ref{sec2} and Section \ref{sec3}).

In \cite{HK24}, one also detects sliding cycles containing a hyperbolic saddle of $Z_{0,c_0}^-$ located in $\{y<0\}$, in such a way that (A3) is violated. The cyclicity of such sliding cycles is a topic of further study. We point out that the cyclicity of canard cycles with hyperbolic saddles located away from curves of singularities has been treated in \cite{H22}. We strongly believe that similar techniques can be used in the PWS setting.

It is important to remark that the regularization process as in systems \eqref{PWS} and \eqref{PWSRescaled} was called linear regularization in \cite{NoJef}. A well-known linear regularization is the so-called {\it Sotomayor--Teixeira regularization} \cite{ST96}. However, different phenomena can arise when considering other regularization processes \cite{LliSilTei06, PanSil,PerRS}. Results concerning the cyclicity of $\Gamma$ with other regularization processes is a topic for future study.

This paper is structured as follows. In Section \ref{sec2} we state our main results about upper bounds for the number of limit cycles of \eqref{PWSRescaled} close to $\Gamma$ in Hausdorff sense. We divide the sliding cycles into different cases, each one depending on the location of singularities. Each case generates different types of cyclicity results. As explained in Section \ref{sec2}, the segment $[\eta_{-},\eta_{+}]$ can contain (a) no singularities of $X_{0,c_0}^{sl}$, (b) singularities of $X_{0,c_0}^{sl}$ located only in the interior $(\eta_{-},\eta_{+})$, (c) a singularity of $X_{0,c_0}^{sl}$ in one of the corner points ($\eta_{-}$ or $\eta_{+}$), (d) singularities of $X_{0,c_0}^{sl}$ in both boundary points, etc. In assumption (A2) we require that singularities located in $(\eta_{-},\eta_{+})$ have even multiplicity because the passage from the left to the right is possible only on one side of singularities of odd multiplicity (thus, $\Gamma$ cannot produce limit cycles if $(\eta_{-},\eta_{+})$ contains a singularity of odd multiplicity). Singularities of odd multiplicity may be located only at the corner points $\eta_{-}$ and $\eta_{+}$. The main results are proven in Section \ref{sec3}.

Finally, in Section \ref{sec4} we apply our results to regularized piecewise linear visible-invisible two-fold singularities \cite{HK24}. In that piecewise linear model, the sliding vector field can have at most $1$ zero counting multiplicity. When such zero exists, Statements 1 and 2 of Theorem \ref{Thm-2} in Section \ref{sec2} will imply that every sliding cycle that satisfies (A1)-(A4) and contains the zero as its corner point can generate at most $1$ limit cycle (for more details we refer the reader to Section \ref{sec4}). Following \cite[Remark 3]{HK24}, there exists a regularized piecewise linear system with $2$ limit cycles produced by sliding cycles not containing the zero of the sliding vector field.

The sliding cycle $\Gamma$ is located in the half-plane
$y\le 0$ with invisible fold point.
Sliding cycles can also appear in the half-plane $y\ge 0$ with a visible fold point (see, e.g. \cite[Figure 2]{HK24}), and their study is analogous to the study of $\Gamma$.

A function $f(\xi,\epsilon)$ is said to be $\epsilon-$regularly $C^k$ in $\xi$ if $f$ is continuous in $(\xi,\epsilon)$, including $\epsilon=0$, and all partial derivatives of $f$ w.r.t. $\xi$ up to order $k$ exist and are continuous in $(\xi, \epsilon)$, including $\epsilon = 0$ (see e.g. \cite{DDR21}). We can take the degree of smoothness $k$ as large as we want, up to shrinking domain. Moreover, if $f$ and all partial derivatives of $f$ w.r.t. $\xi$ up to order $k$ tend to $0$ as $\epsilon\to 0$ (uniformly in $\xi$), then we write $f=o(1)$.

\section{Statement of main results}
\label{sec2}

With assumptions (A1)--(A4) kept in mind, this section is devoted to stating our main results. Before we state them in Section \ref{sec2.2}, we look more carefully at assumptions (A1) and (A2).

\subsection{Assumptions (A1) and (A2)}\label{subsection-assump}
Using the definitions of stable and unstable sliding points and visible-invisible two-fold singularities described in Section \ref{section-intro} (see Figure \ref{fig:PWS-Smooth}(a)), it can be directly checked that the assumption (A1) is equivalent to
\begin{equation}\label{A1-1}
    \begin{array}{ll}
    Y_{0,c}^\pm(0,0)=0, \ \ X_{0,c}^{+}(0,0)>0, \ \  \partial_{x}Y_{0,c}^{+}(0,0)> 0, \ \ X_{0,c}^{-}(0,0)<0,  \ \ \partial_{x}Y_{0,c}^{-}(0,0)<0; \\
    [3mm]
    Y_{0,c}^+(x,0)<0 \ \text{and} \ Y_{0,c}^-(x,0)>0 \ \ \text{if} \ \  x<0, \ \ \text{and} \ \ Y_{0,c}^+(x,0)>0 \ \text{and} \ Y_{0,c}^-(x,0)<0 \ \ \text{if} \ \  x>0,
    \end{array}
\end{equation}
for each $c\sim c_0$. The first line of Equation \eqref{A1-1} means that the origin is a visible-invisible two-fold. The second line of Equation \eqref{A1-1} is the definition of stable and unstable sliding subsets. Moreover, from \eqref{A1-1} it follows that $\partial_{x}(Y_{0,c}^{+}-Y_{0,c}^{-})(0,0)> 0$, and therefore the sliding vector field $X_{0,c}^{sl}$ in \eqref{slvf} has a removable singularity in $x=0$, that is,
$$X_{0,c}^{sl}(x)=\nu(c)+O(x), \ \ \nu(c)=\frac{(\det{Z_{0,c}})'(0)}{\partial_{x}(Y_{0,c}^{+}-Y_{0,c}^{-})(0,0)},$$
where the function $\det{Z_{0,c}}$ is defined in \eqref{slvf-1}. In assumption (A2) we require that $\nu(c_0)>0$, that is, $(\det{Z_{0,c_0}})'(0)>0$, and
$$\det{Z_{0,c_0}}(x)<0 \ \ \text{for} \ \  \eta_{-}\le x<0 \ \ \text{and} \ \ \det{Z_{0,c_0}}(x)>0 \ \  \text{for} \ \  0<x\le \eta_{+},$$
except for a finite number of zeros of $\det{Z_{0,c_0}}$, with finite multiplicity. Let $x_0\in [\eta_{-},\eta_{+}]\setminus\{0\}$. Using \eqref{slvf}, we have that $x_0$ is a zero of multiplicity $n$ of $X_{0,c_0}^{sl}$ if and only if $x_0$ is a zero of multiplicity $n$ of $\det{Z_{0,c_0}}$.

\subsection{Statement of results}\label{sec2.2}

Firstly, we define the notion of cyclicity of the sliding cycle $\Gamma$. From now on, denote the regularized PWS  system \eqref{PWSRescaled} by $Z_{\epsilon,\tilde \lambda,c}$. The reader should not confuse this notation with the components $Z_{\epsilon\tilde\lambda,c}^{\pm}$ of \eqref{PWS-1}. 

\begin{definition}\label{def-cyclicity}
{\rm We say that the cyclicity of $\Gamma$ is bounded by $N\in\mathbb N$ if there exist $\delta>0$, $\epsilon_{0}>0$ and a neighborhood $\mathcal{V}$ of $(0,c_{0})\in\mathbb{R}\times \mathbb{R}^l$ in the $(\tilde\lambda,c)$-space such that, for each $(\epsilon,\tilde \lambda,c)\in ]0,\epsilon_{0})\times \mathcal{V}$, $Z_{\epsilon,\tilde \lambda,c}$ has at most $N$ limit cycles, each of them lying $\delta$-close to $\Gamma$ according to Hausdorff distance. The smallest $N$ with this property is called cyclicity of $\Gamma$ in $Z_{\epsilon,\tilde \lambda,c}$. We denote it by $\cycl(Z_{\epsilon,\tilde \lambda,c},\Gamma)$.}
\end{definition}

The main results of this paper provide upper bounds for the cyclicity of sliding cycles $\Gamma$ in the cases sketched in Figure \ref{FIG3}. We refer to the statements of Theorems \ref{Thm-0}--\ref{Thm-4} for a precise description of each configuration.

\begin{figure}[htp]
\begin{center}
\includegraphics[width=6in]{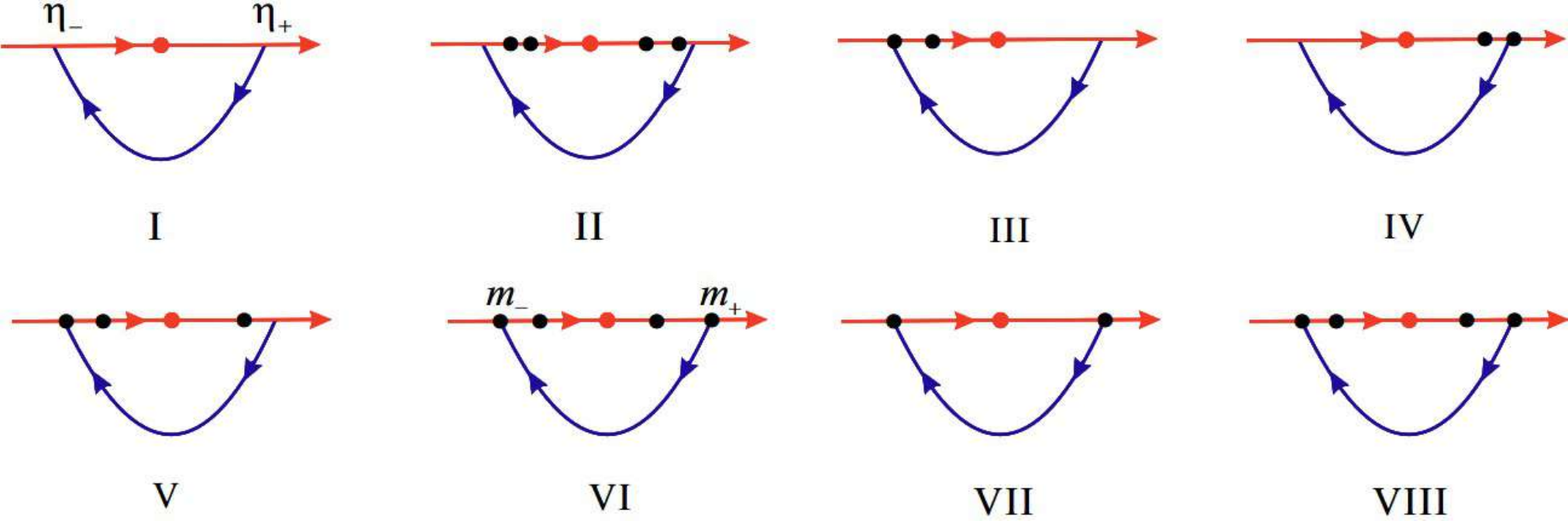}
\caption{\footnotesize{Different types of sliding cycles $\Gamma$, depending on the location of singularities of the sliding vector field $X_{0,c_0}^{sl}$. The red line is the switching manifold $y=0$ and blue curves are orbits of $Z^{-}_{0,c_0}$.
 Red dots are two-folds, black dots represent singularities of $X_{0,c_0}^{sl}$.}}\label{FIG3}
\end{center}
\end{figure}

Denote by $\Pi_c$ the Poincar\'{e} half-map from $\Sigma\cap \{x>0\}$, with $x\sim \eta_{+}$, to $\Sigma\cap \{x<0\}$ by following the orbits of $Z_{0,c}^-$ in forward time. Assumption (A3) implies that $\Pi_c$ is a well-defined $c$-family of smooth diffeomorphisms for $x\sim \eta_{+}$ and $c\sim c_0$, and $\Pi_{c_0}(\eta_{+})=\eta_{-}$.

When the sliding vector field $X_{0,c_0}^{sl}$ is positive on $[\eta_{-},\eta_{+}]\subset\Sigma$, then the following slow divergence integral, computed along the segment $[\Pi_c(x),x]\subset\Sigma$, is well-defined (see \cite[Section 3]{HK23}):
\begin{equation}\label{SDI-statement}
I(x,c)=\int_{\Pi_c(x)}^{x}\frac{(Y_{0,c}^{+}-Y_{0,c}^{-})^{2}(s,0)}{\det Z_{0,c}(s)}\phi'\left(\phi^{-1}\left(\frac{-Y_{0,c}^{-}}{Y_{0,c}^{+}-Y_{0,c}^{-}}(s,0)\right)\right)ds.
\end{equation}
For more details about the slow divergence integral we refer to Section \ref{subsec-e1}.

Theorem \ref{Thm-0} is a direct consequence of \cite{HK23}.

\begin{theorem}
    \label{Thm-0}
    Consider \eqref{PWSRescaled} and suppose that $X_{0,c_0}^{sl}$ is positive on $[\eta_{-},\eta_{+}]\subset\Sigma$ (Case I in Figure \ref{FIG3}). Then the following statements are true.
    \begin{enumerate}
        \item If $I(\eta_{+},c_0)<0$ (resp. $I(\eta_{+},c_0)> 0$), then $\cycl(Z_{\epsilon,\tilde \lambda,c},\Gamma)\le 1$ and the limit cycle, if it exists, is hyperbolic and attracting (resp. repelling).
         \item If $I(x,c_0)$ has a zero of multiplicity $m\ge 1$ at $x=\eta_{+}$, then $\cycl(Z_{\epsilon,\tilde \lambda,c},\Gamma)\le m+1$.
    \end{enumerate}
\end{theorem}
\begin{remark}\label{rem-breaking}
{\rm The main goal of \cite{HK23} was to give a simple criterion (in terms of simple zeros of $I(x,c_0)$) for the existence of limit cycles when the sliding vector field is regular, and the following additional assumption played a key role in \cite{HK23}:
    $$
    \partial_\lambda Y_{\lambda,c}^- \partial_x Y_{\lambda,c}^+\ne \partial_\lambda Y_{\lambda,c}^+ \partial_x Y_{\lambda,c}^-$$
 at $(x,y,\lambda,c)=(0,0,0,c_0)$. We refer to \cite[Theorem 3.1]{HK23}. This assumption is not relevant when we deal with upper bounds for the number of limit cycles produced by $\Gamma$, as in the present paper.}
\end{remark}

In the Theorem \ref{Thm-1} we assume that the sliding vector field $X_{0,c_0}^{sl}$ is regular at corner points $x=\eta_{\pm}$ of the sliding cycle $\Gamma$ and the open interval $(\eta_{-},\eta_{+})$ contains singularities of $X_{0,c_0}^{sl}$ (of even multiplicity). Recall the assumption (A2) and see Case II in Figure \ref{FIG3}. In this case, the slow divergence integral in \eqref{SDI-statement} is not well-defined for $(x,c)=(\eta_{+},c_0)$ due to the presence of singularities in $(\eta_{-},\eta_{+})$ but its derivative
\begin{align}\label{eq-der-SDI}
\frac{\partial I}{\partial x}(x,c)=& \frac{(Y_{0,c}^{+}-Y_{0,c}^{-})^{2}(x,0)}{\det Z_{0,c}(x)}\phi'\left(\phi^{-1}\left(\frac{-Y_{0,c}^{-}}{Y_{0,c}^{+}-Y_{0,c}^{-}}(x,0)\right)\right)\nonumber\\
 &-  \Pi_c'(x)\frac{(Y_{0,c}^{+}-Y_{0,c}^{-})^{2}(\Pi_c(x),0)}{\det Z_{0,c}(\Pi_c(x))}\phi'\left(\phi^{-1}\left(\frac{-Y_{0,c}^{-}}{Y_{0,c}^{+}-Y_{0,c}^{-}}(\Pi_c(x),0)\right)\right)
\end{align}
is a well-defined smooth function in a small neighborhood of $(x,c)=(\eta_{+},c_0)$. Theorem \ref{Thm-1} will be proved in Section \ref{proof-1}.

\begin{theorem}\label{Thm-1}
Suppose that both $X_{0,c_0}^{sl}(\eta_{-})$ and $X_{0,c_0}^{sl}(\eta_{+})$ are nonzero and $(\eta_{-},\eta_{+})$ contains singularities of $X_{0,c_0}^{sl}$ (Case II in Figure \ref{FIG3}). Then $\cycl(Z_{\epsilon,\tilde \lambda,c},\Gamma)\leq 2+m$, where $m$ denotes the multiplicity of zero of $\frac{\partial I}{\partial x}(x,c_0)$ at $x=\eta_{+}$.
\end{theorem}

Theorem \ref {Thm-2} concerns Cases III, IV and V in Figure \ref{FIG3}. More precisely, we assume now that the sliding vector field $X_{0,c_0}^{sl}$ has a singularity at one corner point only, either $x=\eta_{-}$ or $x=\eta_{+}$. Theorem \ref{Thm-2} will be proved in Section \ref{proof-2}.



\begin{theorem}\label{Thm-2} For system \eqref{PWSRescaled}, the following statements hold.
\begin{enumerate}
\item Suppose that $X_{0,c_0}^{sl}(\eta_{-})=0$, $X_{0,c_0}^{sl}(\eta_{+})\ne 0$ and that $X_{0,c_0}^{sl}$ has no singularities in the unstable sliding interval $(0,\eta_{+})$ ($X_{0,c_0}^{sl}$ may have singularities in the stable sliding interval $(\eta_{-},0)$). See Case III in Figure \ref{FIG3}. Then $\cycl(Z_{\epsilon,\tilde \lambda,c},\Gamma)\leq 1$. The limit cycle, if it exists, is hyperbolic and attracting.
\item Suppose that $X_{0,c_0}^{sl}(\eta_{-})\ne 0$, $X_{0,c_0}^{sl}(\eta_{+})=0$ and that $X_{0,c_0}^{sl}$ has no singularities in the stable sliding interval $(\eta_{-},0)$ ($X_{0,c_0}^{sl}$ may have singularities in the unstable sliding interval $(0,\eta_{+})$). See Case IV in Figure \ref{FIG3}. Then $\cycl(Z_{\epsilon,\tilde \lambda,c},\Gamma)\leq 1$. The limit cycle, if it exists, is hyperbolic and repelling.
\item Suppose that $X_{0,c_0}^{sl}(\eta_{-})=0$ and $X_{0,c_0}^{sl}(\eta_{+})\neq 0$ ($X_{0,c_0}^{sl}$ may have singularities in $(\eta_{-},\eta_{+})$). See Case V in Figure \ref{FIG3}. Then $\cycl(Z_{\epsilon,\tilde \lambda,c},\Gamma)\leq 2$. A similar result is true in the case $X_{0,c_0}^{sl}(\eta_{-})\ne 0$, $X_{0,c_0}^{sl}(\eta_{+})=0$.
\end{enumerate}
\end{theorem}

In the rest of this section we assume that both corner points are singularities, that is, $X_{0,c_0}^{sl}(\eta_{-})=X_{0,c_0}^{sl}(\eta_{+})=0$.

Theorem \ref{Thm-3}, whose proof will be given in Section \ref{proof-3}, concerns Case VI in Figure \ref{FIG3}. It addresses the case in which $x=\eta_-$ or $x=\eta_+$ is a non-hyperbolic singularity of $X_{0,c_0}^{sl}$, that is, $(X_{0,c_0}^{sl})'(\eta_{-})=0$ or $(X_{0,c_0}^{sl})'(\eta_{+})=0$.

\begin{theorem}\label{Thm-3}
Consider system \eqref{PWSRescaled} and suppose that $x=\eta_\pm$ is a zero of multiplicity $m_\pm\ge 1$ of $X_{0,c_0}^{sl}$, with $m_-\ne m_+$ (Case VI in Figure \ref{FIG3}). Then $\cycl(Z_{\epsilon,\tilde \lambda,c},\Gamma)\leq 2+\min\{m_-,m_+\}$.
\end{theorem}

Finally, Theorem \ref{Thm-4} concerns Cases VII and VIII in Figure \ref{FIG3}, and its proof will be given in Section \ref{proof-4}. Suppose that $x=\eta_\pm$ are hyperbolic singularities of $X_{0,c_0}^{sl}$. Then, for each $(\epsilon,\tilde \lambda,c)\sim(0,0,c_0)$ and $\epsilon>0$, system \eqref{PWSRescaled} has a hyperbolic saddle near $(x,y)=(\eta_\pm,0)$ with ratio of eigenvalues  $-\epsilon^{2}\rho_{\pm}$, $\rho_{\pm}=\rho_{\pm}(\epsilon,\tilde \lambda,c)>0$ for $\epsilon\ge 0$ (for more details see Section \ref{proof-4}).

\begin{theorem}\label{Thm-4}
Consider system \eqref{PWSRescaled}
and suppose that $X_{0,c_0}^{sl}$ has a hyperbolic singularity at both corner points and that $\rho_{-}\ne \rho_{+}$ for $(\epsilon,\tilde \lambda,c)=(0,0,c_0)$. Then the following statements hold.
\begin{enumerate}
    \item If the interval $(\eta_{-},\eta_{+})$ does not contain singularities of $X_{0,c_0}^{sl}$ (see Case VII in Figure \ref{FIG3}), then $\cycl(Z_{\epsilon,\tilde \lambda,c},\Gamma)\leq 2$.
    \item If the interval $(\eta_{-},\eta_{+})$ contains singularities of $X_{0,c_0}^{sl}$ (see Case VIII in Figure \ref{FIG3}), then $\cycl(Z_{\epsilon,\tilde \lambda,c},\Gamma)\leq 3$.
\end{enumerate}
\end{theorem}

Notice that the case $m_-=m_+$ (resp. $\rho_{-}= \rho_{+}$) is not covered by Theorem \ref{Thm-3} (resp. Theorem \ref{Thm-4}). We leave this to future work.

\section{Proof of main results}
\label{sec3}

The objective of this section is to prove Theorems \ref{Thm-1}-\ref{Thm-4}. Firstly, in Subsection \ref{subsection-bu} we blow up system \eqref{PWSRescaled} along the switching manifold, and afterwards in Subsection \ref{section-transition-maps} we study properties of the transition maps of vector fields defined on different charts. Then we use it to reduce the problem of finding the number of limit cycles produced by the sliding cycle $\Gamma$ to the problem of studying the solutions of equation \eqref{equ-most-import}. The main reason why we work with \eqref{equ-most-import} is that we can then use the results of \cite{DD08} (for more details, we refer the reader to Remark \ref{remark-bitno} and Remark \ref{remark-bitno-new}). Finally, in Subsections \ref{proof-1}-\ref{proof-4} we prove Theorems \ref{Thm-1}-\ref{Thm-4}.

\subsection{Cylindrical blow-up}
\label{subsection-bu}

Following the approach of \cite{HK23}, we write system \eqref{PWSRescaled} as the extended vector field
\begin{equation}\label{extended}
Z_{\epsilon,\tilde\lambda,c} + 0\frac{\partial }{\partial \epsilon}.
\end{equation}

We blow up system \eqref{extended} along the switching manifold $\Sigma$ by using the singular coordinate change
\begin{equation}\label{eq-cylin-blowup}
  \begin{array}{rccc}
\Phi: & \mathcal{M} & \rightarrow & \mathbb{R}^{3} \\
& (x,\bar{y},\bar{\epsilon},r) & \mapsto & (x, r^{2}\bar{y}, r\bar{\epsilon}) = (x,y,\epsilon);
\end{array}
\end{equation}
with $\mathcal{M}$ being a manifold with corners, $(\bar{y},\bar{\epsilon})\in\mathbb{S}^{1}$, $r\geq 0$ and $\bar{\epsilon}\geq 0$. Such transformation induces a vector field $\overline{Z}$ on $\mathcal{M}$, which is the pull-back of \eqref{extended} by \eqref{eq-cylin-blowup}. We then study the vector field $\widehat{Z} := r^{2}\overline{Z}$. For studying the cyclicity of sliding cycles, it suffices to investigate the dynamics of $\widehat{Z}$ on the charts $\{\bar{\epsilon}=1\}$ and $\{\bar{y}=-1\}$. See Figure \ref{FIG4}.

\begin{figure}[htp]
\begin{center}
\includegraphics[width=4in]{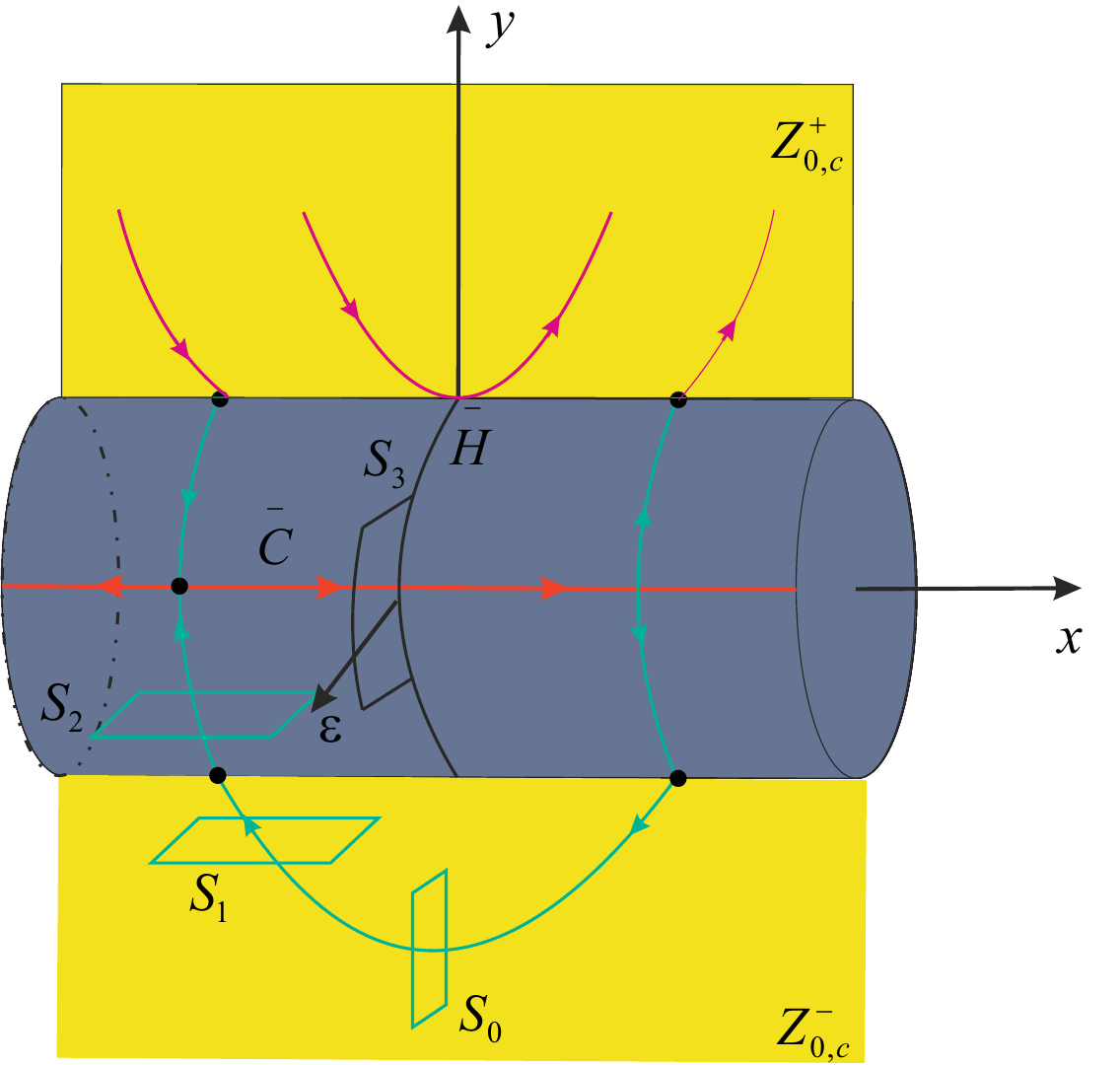}
\caption{\footnotesize{Cylindrical blow-up along $(\Sigma_{sl}\times \{0\})\cup\{(0,0,0)\}$ for the extended system \eqref{extended}. We use the case III of Figure \ref{FIG3} as an example, where $\eta_{-}$ is a singularity of $X_{0,c_0}^{sl}(x)$ with odd multiplicity.}}\label{FIG4}
\end{center}
\end{figure}

\subsubsection{Dynamics in the chart $\{\bar{\epsilon}=1\}$}\label{subsec-e1}
The transformation takes the form $(x,y,\epsilon)=(x_2,r_{2}^{2}y_2,r_2)$, with coordinates $(x_2,y_2,r_2)$. We bring the transformation into system \eqref{extended}, and, after multiplication by $r_{2}^{2}$, we have a slow-fast system
\begin{equation}\label{family}
\begin{array}{rcl}
\dot{x}_{2} & = & r_{2}^{2}\left(X^{+}_{r_2\tilde \lambda, c}(x_2,r_{2}^{2}y_2)\phi(y_2) + X^{-}_{r_2\tilde \lambda, c}(x_2,r_{2}^{2}y_2)\left(1-\phi(y_2)\right)\right), \\[3mm]
\dot{y}_{2} & = & Y^{+}_{r_2\tilde\lambda, c}(x_2,r_{2}^{2}y_2)\phi(y_2)+Y^{-}_{r_2\tilde\lambda, c}(x_2,r_{2}^{2}y_2)\left(1-\phi(y_2)\right).
\end{array}
\end{equation}

Setting $r_2=0$ in \eqref{family}, we obtain
\begin{equation}\label{family-0}
\dot{x}_2 = 0, \quad \dot{y}_2 = Y^{+}_{0,c}(x_2,0)\phi(y_2)+Y^{-}_{0,c}(x_2,0)\left(1-\phi(y_2)\right).
\end{equation}

From Assumption (A1) (see also Equation \eqref{A1-1}), it follows that there are two branches of curve of singularities (see Figure \ref{FIG4})
$$\overline{H}:=\{(0,y_2)\mid y_2\in \mathbb{R}\}, \quad \overline{C}:= \left\{(x_2,y_2)\mid y_2 = \phi^{-1}\left(\frac{-Y^{-}_{0,c}}{Y^{+}_{0,c}-Y^{-}_{0,c}}(x_2,0)\right)\right\},$$
which intersect transversally at the point $p_0=\left(0,\phi^{-1}\left(\frac{-\partial_{x} Y_{0,c}^{-}}{\partial_{x}Y_{0,c}^{+}-\partial_{x}Y_{0,c}^{-}}(0,0)\right)\right)$.

The next statement summarizes important features about the dynamics of \eqref{family} close to $\overline{C}$ (see also \cite[Theorem 2.2]{HK23}). We refer to \cite{BST06} for an analogous statement in a more general context.

\begin{proposition}\label{prop-slow-sliding}
The following statements are true.
\begin{enumerate}
    \item The singularities on $\overline{C}$ are normally attracting (resp. repelling) when $x_{2} < 0$ (resp. $x_{2} > 0$).
    \item For each $0 < r_2 \ll 1$, $\tilde\lambda\sim 0$ and $c\sim c_0$, there is a locally invariant manifold of \eqref{family} with foliation by stable fibers (resp. unstable fibers), close to the normally attracting (resp. repelling) branch of $\overline{C}$.
    \item The slow dynamics of \eqref{family} along $\overline{C}$ is given by $\frac{d{x}_{2}}{dt} = X^{sl}_{0,c}(x_2)$, where the sliding vector field $X^{sl}_{0,c}$ is defined in \eqref{slvf}.
\end{enumerate}
\end{proposition}
\begin{proof}
For $x_2\ne 0$, each singularity $(x_2,y_2)\in \overline{C}$ is normally hyperbolic with the nontrivial eigenvalue equal to the divergence of the vector field \eqref{family-0}, computed in $(x_2,y_2)$, that is,
$$\left(Y^{+}_{0,c}(x_2,0) - Y^{-}_{0,c}(x_2,0)\right)\phi'(y_2).$$

From \eqref{regular} it follows that $\phi'(y_2)>0$, and the sign of the eigenvalue therefore depends on the sign of $Y^{+}_{0,c}(x_2,0) - Y^{-}_{0,c}(x_2,0)$. Using the assumption (A1), or equivalently \eqref{A1-1}, we have that $\overline{C}$ is normally attracting (resp. repelling) when $x_{2} < 0$ (resp. $x_{2} > 0$). This shows that Statement $1$ is true. Statement $2$ is a direct consequence of  Fenichel’s theory \cite{Fen79}.

Finally, we prove Statement $3$. Invariant manifolds from Statement 2 of Proposition \ref{prop-slow-sliding} are asymptotic to the curve of singular points $\overline{C}$, and we can write
$$y_{2} = \phi^{-1}\left(\frac{-Y^{-}_{0,c}}{Y^{+}_{0,c}-Y^{-}_{0,c}}(x_2,0)\right) + O(r_{2}).$$

If we reduce the dynamics of \eqref{family} to the invariant manifolds, divide out $r_{2}^{2}$ and finally let $r_{2}\rightarrow 0$, we get the slow dynamics \cite[Chapter 3]{DDR21} of \eqref{family} along $\overline{C}$ (away from $p_0$) which is exactly the flow of the sliding vector field $X^{sl}_{0,c}(x_2)$. This completes the proof of Statement 3.
\end{proof}

\begin{remark}\label{remark-bitno} {\rm Statement 3 of Proposition \ref{prop-slow-sliding} and the assumption (A2) (see also Section \ref{subsection-assump}) imply that for $c=c_0$ the slow dynamics of \eqref{family} along the portion of $\overline{C}$ parameterized by $x_2\in [\eta_{-},\eta_{+}]$ is positive, except at a finite number of singularities of the slow dynamics located in $[\eta_{-},\eta_{+}]\setminus\{0\}$ and with finite multiplicity (each singularity located in the interior $(\eta_{-},\eta_{+})$ has even multiplicity). Notice that the slow dynamics is regularly extended through the origin $x_2=0$. This brings the slow-fast system \eqref{family} into the framework of \cite{DD08}. In Appendix \ref{appendix-extra-blowup} we show that the assumptions in \cite{DD08} near the intersection point $p_0$ are also satisfied. We use a cylindrical blow-up along the curve of singularities $\overline{H}$.}
\end{remark}

If the slow dynamics (or, equivalently, the sliding vector field) is regular along the attracting segment $[\eta_-,0]\subset \overline C$, for $c=c_0$, then we can define the associated slow divergence integral \cite[Chapter 5]{DDR21}
\begin{equation}\label{slowdivint-attracting}
    I_-(x,c)=\int_{x}^{0}\frac{(Y_{0,c}^{+}-Y_{0,c}^{-})^{2}(s,0)}{\det Z_{0,c}(s)}\phi'\left(\phi^{-1}\left(\frac{-Y_{0,c}^{-}}{Y_{0,c}^{+}-Y_{0,c}^{-}}(s,0)\right)\right)ds<0,
\end{equation}
with $x\sim\eta_{-}$ and $c\sim c_{0}$. This is the integral of the divergence of the vector field \eqref{family-0} calculated at points of the singular curve $\overline{C}$, where the integration variable is the time variable of the flow of the slow vector field (which is $dt=\frac{ds}{X^{sl}_{0,c}(s)}$).

Similarly, if the slow dynamics is regular along the repelling segment $[0,\eta_+]\subset \overline C$, then slow divergence integral is well defined
\begin{equation}\label{slowdivint-repelling}
    I_+(x,c)=\int_{x}^{0}\frac{(Y_{0,c}^{+}-Y_{0,c}^{-})^{2}(s,0)}{\det Z_{0,c}(s)}\phi'\left(\phi^{-1}\left(\frac{-Y_{0,c}^{-}}{Y_{0,c}^{+}-Y_{0,c}^{-}}(s,0)\right)\right)ds<0,
\end{equation}
with $x\sim\eta_{+}$ and $c\sim c_{0}$. It is clear that
$$I(x,c)=I_-(\Pi_c(x),c)-I_+(x,c),$$
where $I$ is defined in \eqref{SDI-statement}.



\subsubsection{Dynamics in the chart $\{\bar{y}=-1\}$}\label{subsec-y1}
The transformation takes the form $(x,y,\epsilon)=(x_{1},-r_{1}^2, r_{1}\epsilon_{1})$, where we use new coordinates $(x_1,r_1,\epsilon_{1})$ to distinguish the coordinates of the chart $\{\bar{\epsilon}=1\}$. We bring the transformation into system \eqref{extended}, which becomes (after multiplication by $r_{1}^{2}$)
\begin{equation}\label{phase}
\begin{array}{rcl}
\dot{x}_1 & = & r_{1}^{2}\left(X^{+}_{r_{1}\epsilon_{1} \tilde \lambda, c}(x_1,-r_{1}^{2})\phi_{-}(\epsilon_{1}^{2})+X^{-}_{r_{1}\epsilon_{1} \tilde \lambda, c}(x_1,-r_{1}^{2})(1-\phi_{-}(\epsilon_{1}^{2}))\right), \\[3mm]
\dot{r}_1 & = & -\frac{1}{2}r_{1}\left(Y^{+}_{r_{1}\epsilon_{1} \tilde \lambda, c}(x_1,-r_{1}^{2})\phi_{-}(\epsilon_{1}^{2})+Y^{-}_{r_{1}\epsilon_{1} \tilde \lambda, c}(x_1,-r_{1}^{2})(1-\phi_{-}(\epsilon_{1}^{2}))\right),  \\[3mm]
\dot{\epsilon}_1 & = & \frac{1}{2}\epsilon_{1}\left(Y^{+}_{r_{1}\epsilon_{1} \tilde \lambda, c}(x_1,-r_{1}^{2})\phi_{-}(\epsilon_{1}^{2})+Y^{-}_{r_{1}\epsilon_{1} \tilde \lambda, c}(x_1,-r_{1}^{2})(1-\phi_{-}(\epsilon_{1}^{2}))\right),
\end{array}
\end{equation}
where $\phi_{-}$ is defined in \eqref{smooth}. It is easily shown that the $x_{1}$-axis (which corresponds to the edge of the cylinder) is a curve of singularities whose eigenvalues are given by $(0,-\frac{Y^-_{0,c}(x_1,0)}{2},\frac{Y^-_{0,c}(x_1,0)}{2})$, for any singularity $(x_1,0,0)$. We can refer to \cite[Section 3.1.2]{HK23} for more details.

\subsection{Transition maps and their expressions}\label{section-transition-maps}

\subsubsection{The difference map}\label{subsection-DM}
In order to study the cyclicity of the sliding cycle $\Gamma$ (see Definition \ref{def-cyclicity} in Section \ref{sec2.2}), we first define the associated difference map by following the trajectories of the blown-up vector field $\widehat{Z}$ in forward and backward time.

Assumption (A3) in Section \ref{section-intro} guarantees the existence of a point $(0,s_0)\in\{x = 0\}$, which is the intersection between the line $\{x = 0\}$ and the regular orbit of $Z_{0,c_0}^{-}$ connecting $(\eta_{+},0)$ and $(\eta_{-},0)$. Consider a sufficiently small segment $[\sigma_1,\sigma_2]$ containing $s_0$ in its interior $(\sigma_1,\sigma_2)$. Define a section $S_0\subset\{x=0\}$ parametrized by $\epsilon\in[0,\epsilon_0]$ and $y\in[\sigma_1,\sigma_2]$ where $\epsilon_0>0$ is small enough (see Figure \ref{FIG4}).
In the chart $\{\bar{\epsilon} = 1\}$ studied in Section \ref{subsec-e1}, we define another section $S_3\subset \{x_2=0\}$ (see Figure \ref{FIG4}), parametrized by $r_2=\epsilon\in[0,\epsilon_0]$ and $y_2\in[\upsilon_1,\upsilon_2]$, where $[\upsilon_1,\upsilon_2]$ is a segment that contains the point $y_2=\phi^{-1}\left(\frac{-\partial_{x} Y_{0,c_0}^{-}}{\partial_{x}Y_{0,c_0}^{+}-\partial_{x}Y_{0,c_0}^{-}}(0,0)\right)$ in its interior $(\upsilon_1,\upsilon_2)$.

Finally, we define the difference map
\begin{equation*}
\Delta(y,\epsilon,\tilde{\lambda},c)=\Delta_{+}(y,\epsilon,\tilde{\lambda},c) - \Delta_{-}(y,\epsilon,\tilde{\lambda},c),
\end{equation*}
for $\epsilon\in(0,\epsilon_0], \ \tilde{\lambda}\sim 0$ and $c\sim c_0$, where $y_2=\Delta_{-}(y,\epsilon,\tilde{\lambda},c)$ (resp. $y_2=\Delta_{+}(y,\epsilon,\tilde{\lambda},c)$) is the transition map between $S_0$ and $S_3$ defined by following the trajectories of the blown-up vector field $\widehat{Z}$ in forward (resp. backward) time. It is clear that the zeros of the difference map $\Delta$ correspond to periodic orbits of system \eqref{PWSRescaled}, Hausdorff close to $\Gamma$. It will be more convenient to study the zeros of the derivative of $\Delta$ w.r.t. $y$ and use Rolle's Theorem (see later sections).

For our purposes, we will write $\Delta_{-} = \Delta_{23}\circ \Delta_{12} \circ \Delta_{01}$, in which the functions $\Delta_{01}$, $\Delta_{12}$ and $\Delta_{23}$ will be defined and studied in Sections \ref{subsection-del01}--\ref{subsection-Del23}. In Section \ref{subsection-combi} we combine the local results from Sections \ref{subsection-del01}--\ref{subsection-Del23} and obtain an expression for the derivative of $\Delta_-$ w.r.t. $y$. The transition map $\Delta_{+}$ can be treated in a similar fashion.

\subsubsection{Transition map $\Delta_{01}$}\label{subsection-del01}
Consider the chart $\{\bar{y} = -1\}$ with coordinates $(x_1,r_1,\epsilon_1)$ studied in Section \ref{subsec-y1}. For a sufficiently small $r_{10} > 0$, define a section $S_{1}\subset\{r_{1}=r_{10}\}$ parametrized by $x = x_1\in[\eta_1,\eta_2]$ and $\epsilon_1\in[0,\frac{\epsilon_0}{r_{10}}]$. We assume that the orbit of $Z_{0,c_0}^{-}$ connecting $(\eta_{+},0)$ and $(\eta_{-},0)$ intersects $S_1$ for some $x\in(\eta_1,\eta_2)$ (see Figure \ref{FIG4}). Denote by $x=\Delta_{01}(y,\epsilon,\tilde{\lambda},c)$ the transition map from $S_0$ (see Section \ref{subsection-DM}) to $S_1$. It is clear that $\epsilon_1=\frac{\epsilon}{r_{10}}$.

\begin{lemma}\label{L1}
The transition map $\Delta_{01}$ is smooth with respect to $(y,\epsilon,\tilde{\lambda},c)$ and
\begin{equation}\label{map01}
\frac{\partial \Delta_{01}}{\partial y}(y,\epsilon,\tilde{\lambda},c)>0.
\end{equation}
\end{lemma}

\begin{proof} The smoothness of $\Delta_{01}$ follows directly from \eqref{phase} and the fact that the vector field $Z^{-}_{\epsilon\tilde\lambda,c}$ is regular between $S_0$ and $S_1$ (see Assumption (A3) in Section \ref{section-intro}). See also \cite[Proposition 3.2]{HK23}. The chosen parametrization of $S_0$ and $S_1$ implies \eqref{map01}.
\end{proof}

\subsubsection{Transition map $\Delta_{12}$}\label{subsection-Del12}

We still work in the chart $\{\bar{y} = -1\}$ with coordinates $(x_1,r_1,\epsilon_1)$, as studied in Section \ref{subsec-y1}. For a small positive constant $\epsilon_{10}$, define a section $S_2\subset\{\epsilon_{1}=\epsilon_{10}\}$, parametrized by $x=x_1\in[\bar\eta_1,\bar\eta_2]$ and $r_{1}\in[0,\frac{\epsilon_0}{\epsilon_{10}}]$, with $\eta_{-}\in(\bar \eta_1,\bar \eta_2)$ (see Figure \ref{FIG4}). We denote by $\Delta_{12}(x,\epsilon_1,\tilde{\lambda},c)$ the $x_1$-component of the transition map from $S_1$ (see Section \ref{subsection-del01}) to $S_2$, following the trajectories of system \eqref{phase} inside the level $\{r_1\epsilon_1 = \epsilon\}$ for any fixed $\epsilon\in(0,\epsilon_0]$. Clearly, the $r_1$-component of the transition map is given by $r_1=\frac{\epsilon}{\epsilon_{10}}=\frac{r_{10}\epsilon_1}{\epsilon_{10}}$, where $r_{10}$ is defined in Section \ref{subsection-del01}. A proof of the next result can be found in \cite[Appendix B]{HK23}.

\begin{lemma}\label{L2}
The transition map $\Delta_{12}$ is given by
$$\Delta_{12}\left(x,\frac{\epsilon}{r_{10}},\tilde{\lambda},c\right)=g_{12}(x,\tilde{\lambda},c)+O(\epsilon\ln\epsilon^{-1})$$ as $\epsilon\rightarrow 0$, where both $g_{12}(x,\tilde{\lambda},c)$ and $O(\epsilon\ln\epsilon^{-1})$ are smooth in $(x,\tilde{\lambda},c)$. This expression can be differentiated with respect to $x,\tilde{\lambda},c$ without changing the order of the remainder. In particular,
$$g_{12}\left(\Delta_{01}(s_0,0,0,c_0),0,c_0\right)=\eta_{-} \quad \text{and}  \quad \frac{\partial g_{12}}{\partial x}(x,\tilde{\lambda},c)>0,$$
with $s_0$ defined in Section \ref{subsection-DM}.
\end{lemma}

\begin{remark}
{\rm Notice that $\frac{\partial g_{12}}{\partial x}$ is positive due to the chosen parametrization of $S_1$ and $S_2$.}
\end{remark}

\subsubsection{Transition map $\Delta_{23}$}\label{subsection-Del23}
In this section we study the $y_2$-component $y_2=\Delta_{23}(x,\epsilon,\tilde\lambda,c)$ of the transition map from $S_2$ (defined in Section \ref{subsection-Del12}) to $S_3$ (defined in Section \ref{subsection-DM}) w.r.t. the slow-fast system \eqref{family}. We denote \eqref{family} by $\widehat Z_{\epsilon,\tilde\lambda,c}$ where we use $\epsilon=r_2$ and write $\epsilon$ instead of $r_2$ (see Section \ref{subsec-e1}).


The following lemma gives a well-known connection between the divergence integral along orbits and transition maps for planar vector fields (see e.g. \cite[Theorem 5.1]{DDR21}).

\begin{lemma}\label{Lemma3}
Let $X=(X_1,X_2)$ be a vector field on $\mathbb{R}^2$. Let $\bar S_1$, $\bar S_2$ be two sections of $\mathbb{R}^2$, transverse to the flow of $X$. Assume $p\in \bar S_1$, $q\in \bar S_2$ and the orbit through $p$ reaches $q$ in finite time. Let $T:\tilde S_1\subset \bar S_1\rightarrow \bar S_2$ be the transition map defined in a neighborhood of $p$. If $\Psi_{i}=(\Psi_{i,x},\Psi_{i,y}):U_{i}\rightarrow \bar S_{i}$  are coordinates for $\bar S_{i}$ with $U_{i}\subset \mathbb{R}$, then
$$(\Psi_{2}^{-1}\circ T \circ\Psi_{1})'(s_1)=\frac{\begin{vmatrix} \Psi_{1,x}'(s_1) & X_1(\Psi_{1}(s_1))  \\ \Psi_{1,y}'(s_1) & X_2(\Psi_{1}(s_1))  \end{vmatrix} }{\begin{vmatrix} \Psi_{2,x}'(s_2) & X_1(\Psi_{2}(s_2))  \\ \Psi_{2,y}'(s_2) & X_2(\Psi_{2}(s_2))  \end{vmatrix}  }\cdot\exp\left( \int_{\mathcal{O}(p,q)}\divergenceOperator Xdt\right),$$
where $s_1=\Psi_{1}^{-1}(p)$, $s_2=\Psi_{2}^{-1}(q).$ The integral is taken over the orbit $\mathcal{O}(p,q)$ from $p$ to $q$ parameterized by the time.
\end{lemma}

Before we study the derivative of the transition map $\Delta_{23}$, define the integral
\begin{equation}\label{eq-divergence-intg}
\widetilde{I}_{-}(x,\epsilon,\tilde\lambda,c) :=
 \epsilon^2\int_{\mathcal{O_{-}}(x,\epsilon,\tilde\lambda,c)} \divergenceOperator \widehat Z_{\epsilon,\tilde\lambda,c}dt,
\end{equation}
where $\mathcal{O_{-}}(x,\epsilon,\tilde\lambda,c)$ is the orbit of the vector field $\widehat Z_{\epsilon,\tilde\lambda,c}$ (which is the slow-fast system \eqref{family}) through the point $p = (x, -\frac{1}{\epsilon_{10}^2})\in S_2$ until it hits the section $S_{3}$ at $q = (0,\Delta_{23}(x,\epsilon,\tilde\lambda,c))$ (following the notation of Lemma \ref{Lemma3}). It is important to remark that the integral \eqref{eq-divergence-intg} and the slow divergence integral \eqref{slowdivint-attracting} are different.

\begin{lemma}\label{L4}
The derivative of the transition map $\Delta_{23}$ w.r.t. $x$ is given by
\begin{equation}
\frac{\partial \Delta_{23}}{\partial x}(x,\epsilon,\tilde\lambda,c)=-\exp\frac{1}{\epsilon^2}\left(\widetilde{I}_{-}(x,\epsilon,\tilde\lambda,c) + o(1)\right),
\end{equation}
where $o(1)$ is $\epsilon-$regularly $C^{k}$ in $(x,\tilde{\lambda},c)$ and tends to $0$ as $\epsilon\to 0$, and $\widetilde{I}_{-}(x,\epsilon,\tilde\lambda,c)$ is the integral \eqref{eq-divergence-intg}.
\end{lemma}
\begin{proof} For simplicity we denote $\Delta_{23}(x)=\Delta_{23}(x,\epsilon,\tilde\lambda,c)$. Applying Lemma \ref{Lemma3} in \eqref{family} ($x = x_1$ is a coordinate of $S_2$ and $y_2$ is a coordinate of $S_3$) it follows that
\begin{align}
\frac{\partial \Delta_{23}}{\partial x}(x,\epsilon,\tilde\lambda,c)=-\epsilon^{-2}B(x,\epsilon,\tilde\lambda,c)\exp\left(\int_{\mathcal{O_{-}}(x,\epsilon,\tilde\lambda,c)} \divergenceOperator \widehat Z_{\epsilon,\tilde\lambda,c}dt\right),
\end{align}
where $\mathcal{O_{-}}(x,\epsilon,\tilde\lambda,c)$ is the orbit of the vector field $\widehat Z_{\epsilon,\tilde\lambda,c}$ through the point $(x,\frac{\epsilon}{\epsilon_{10}})\in S_2$ until it hits the section $S_3$ and
$$B(x,\epsilon,\tilde\lambda,c) := \frac{Y^{+}_{\epsilon\tilde\lambda, c}(x,-\frac{\epsilon^{2}}{\epsilon_{10}^2})\phi(-\frac{1}{\epsilon_{10}^2})+Y^{-}_{\epsilon\tilde\lambda, c}(x,-\frac{\epsilon^{2}}{\epsilon_{10}^2})\left(1-\phi(-\frac{1}{\epsilon_{10}^2})\right)}{X^{+}_{\epsilon\tilde \lambda, c}(0,\epsilon^{2}\Delta_{23}(x))\phi(\Delta_{23}(x)) + X^{-}_{\epsilon\tilde \lambda, c}(0,\epsilon^{2}\Delta_{23}(x))\left(1-\phi(\Delta_{23}(x))\right)},$$
where $\epsilon_{1}=\epsilon_{10}$ corresponds to $y_2=-\frac{1}{\epsilon_{10}^2}$.

By checking the direction of the flow of system \eqref{family} near the section $S_2$ (see Figure \ref{FIG4}), it is clear that the numerator of $B$ is positive. Remark \ref{remark-bitno} and \cite[Theorem 2.16]{DD08} imply that $\Delta_{23}$ is smooth and has a $C^k$-extension to the closure of its domain. Moreover, $\Delta_{23}$ tends to the $y_2$-component of the intersection point $p_0$ defined in Section \ref{subsec-e1} as $\epsilon\to 0$. This implies that the denominator of $B$ tends to $\nu(c)$ as $\epsilon\to 0$, where $\nu(c)$ is defined in Section \ref{subsection-assump}. Since $\nu(c_0)>0$, the denominator is also positive. Thus, the function $B$ is positive, bounded from above, bounded from below and away from zero.

Now, we have
\begin{equation}\label{dominant-1}
\begin{split}
\frac{\partial \Delta_{23}}{\partial x}(x,\epsilon,\tilde\lambda,c)&=-\exp\left(\ln(\epsilon^{-2}B(x,\epsilon,\tilde\lambda,c))\right)\cdot\exp\left(\int_{\mathcal{O_{-}}(x,\epsilon,\tilde\lambda,c)} \divergenceOperator \widehat Z_{\epsilon,\tilde\lambda,c}dt\right)
\\&=-\exp\frac{1}{\epsilon^2}\left(\widetilde{I}_{-}(x,\epsilon,\tilde\lambda,c) + \epsilon^{2}\ln B(x,\epsilon,\tilde\lambda,c) - 2\epsilon^{2}\ln \epsilon\right)
\\&=-\exp\frac{1}{\epsilon^2}\left(\widetilde{I}_{-}(x,\epsilon,\tilde\lambda,c) + o(1)\right),\nonumber
\end{split}
\end{equation}
where $o(1)$ is $\epsilon-$regularly $C^{k}$ in $(x,\tilde{\lambda},c)$ and tends to $0$ as $\epsilon\to 0$.
Notice that $\epsilon^{2}\ln \epsilon\rightarrow 0$ as $\epsilon\rightarrow 0$. This completes the proof of Lemma \ref{L4}.
\end{proof}

\subsubsection{Combining $\Delta_{01}$, $\Delta_{12}$ and $\Delta_{23}$}\label{subsection-combi}

Recall that $\Delta_-$ is the forward transition map between $S_0$ and $S_3$ introduced in Section \ref{subsection-DM}. For simplicity sake, we write $\Delta_{-}(y)=\Delta_{-}(y,\epsilon,\tilde{\lambda},c)$ and $\widetilde{I}_{-}(x) = \widetilde{I}_{-}(x,\epsilon,\tilde\lambda,c)$. We also write $\Delta_{01}(y)=\Delta_{01}(y,\epsilon,\tilde{\lambda},c)$, $\Delta_{12}(x)=\Delta_{12}(x,\frac{\epsilon}{r_{10}},\tilde{\lambda},c)$, $\Delta_{23}(x)=\Delta_{23}(x,\epsilon,\tilde\lambda,c)$ and finally $\Delta_{-}=\Delta_{23}\circ \Delta_{12} \circ \Delta_{01}$. Additionally, denote $\Delta_{02}^-(y)=\Delta_{12}\circ\Delta_{01}(y)$.

Using Lemmas \ref{L1}, \ref{L2}, \ref{L4} and the chain rule, we have
\begin{equation*}
\begin{array}{rcl}
  \Delta_{-}'(y) & = & -\exp\displaystyle\frac{1}{\epsilon^2}\left(\widetilde{I}_{-}(\Delta_{02}^-(y)) + o(1)\right) \cdot\left(\frac{\partial g_{12}}{\partial x}(\Delta_{01}(y),\tilde{\lambda},c)+O(\epsilon\ln\epsilon^{-1})\right)\cdot\Delta_{01}'(y)\\
  & =  & -\exp\displaystyle\frac{1}{\epsilon^2}\bigg(\widetilde{I}_{-}(\Delta_{02}^-(y)) + o(1) + \epsilon^2 \ln\left(\displaystyle\frac{\partial g_{12}}{\partial x}(\Delta_{01}(y),\tilde{\lambda},c)+O(\epsilon\ln\epsilon^{-1})\right)+\epsilon^2\ln\Delta_{01}'(y)\bigg)
\\& = & -\exp\displaystyle\frac{1}{\epsilon^2}\left(\widetilde{I}_{-}(\Delta_{02}^-(y)) + o(1)\right),
  \end{array}
\end{equation*}
where $o(1)$ is $\epsilon-$regularly $C^{k}$ in $(y,\tilde{\lambda},c)$ tending to $0$ as $\epsilon\to 0$. We used the fact that $\Delta_{01}'>0$ (Lemma \ref{L1}) and $\frac{\partial g_{12}}{\partial x}>0$ (Lemma \ref{L2}). Note also that $\Delta_{01}'$ and $\frac{\partial g_{12}}{\partial x}$ are bounded and uniformly away from zero.

We can find a similar expression for the backward transition map $\Delta_{+}(y)=\Delta_{+}(y,\epsilon,\tilde{\lambda},c)$ between $S_0$ and $S_3$ defined in Section \ref{subsection-DM}:
\begin{equation*}
    \Delta_{+}'(y) = -\exp\displaystyle\frac{1}{\epsilon^2}\left(\widetilde{I}_{+}(\Delta_{02}^{+}(y)) + o(1)\right), \quad \text{with} \quad \widetilde{I}_{+}(x) :=
 \epsilon^2\int_{\mathcal{O_{+}}(x,\epsilon,\tilde\lambda,c)} \divergenceOperator \widehat Z_{\epsilon,\tilde\lambda,c}dt,
\end{equation*}
where $o(1)$ is $\epsilon-$regularly $C^{k}$ in $(y,\tilde{\lambda},c)$ tending to $0$ as $\epsilon\to 0$, $\mathcal{O_{+}}(x,\epsilon,\tilde\lambda,c)$ is the orbit through the point $(x,\frac{\epsilon}{\epsilon_{10}})\in S_2'$ in negative time until it reaches the section $S_3$, and $\Delta_{02}^+(y)$ is the backward transition map from $S_0$ to $S_2'$ ($S'_{2}\subset\{\epsilon_{1}=\epsilon_{10}\}$ has similar parametrization as $S_2$ using $x>0$ instead of $x<0$). In fact, $\Delta_{02}^+(y)$ depends on $(y,\epsilon,\tilde\lambda,c)$. However, we adopt the notation $\Delta_{02}^+(y)$ for simplicity.

Using the above expressions for $\Delta_{\pm}'(y)$, we can easily see that, for $\epsilon>0$, the equation $\frac{\partial \Delta}{\partial y}(y,\epsilon,\tilde\lambda,c)=0$ is equivalent to
\begin{equation}\label{equ-most-import}
    \widetilde{I}_{-}(\Delta_{02}^{-}(y)) - \widetilde{I}_{+}(\Delta_{02}^{+}(y)) + o(1)=0,
\end{equation}
where $o(1)$ has the same property as above. As a simple consequence of Rolle's Theorem, for fixed $(\epsilon,\tilde\lambda,c)$ the number of periodic orbits Hausdorff close to the sliding cycle $\Gamma$ is bounded by the number of solutions of \eqref{equ-most-import} w.r.t. $y$, increased by $1$. Equation \eqref{equ-most-import} plays an important role in the proof of Theorems \ref{Thm-1}-\ref{Thm-4}.

\begin{remark}\label{remark-bitno-new}
{\rm The integrals $\widetilde{I}_{\pm}$ have been studied in \cite{DD08} (see also Remark \ref{remark-bitno}). This fact will significantly simplify the proofs in Sections \ref{proof-1}-\ref{proof-4}. Since we deal with $\widetilde{I}_{\pm}(\Delta_{02}^\pm(y),\epsilon,\tilde\lambda,c)$ in \eqref{equ-most-import}, we will also use Lemma \ref{L1} and Lemma \ref{L2}. Notice that the $o(1)$-term in \eqref{equ-most-import} is not specified because it is not the leading-order part in \eqref{equ-most-import} (for more details, we refer to Sections \ref{proof-1}-\ref{proof-4}).
}
\end{remark}

\begin{remark}\label{remark-bitno-new2} {\rm
When $\epsilon=0$, we denote $\Delta_{02}^\pm(y)$ by $\Pi^\pm_c(y)$. It is clear that $\Pi^-_c(y)$ (resp. $\Pi^+_c(y)$) is the $x$-value of the intersection with the switching line $\{y=0\}$ of the forward (resp. backward) flow of $(0,y)$ w.r.t. $Z_{0,c}^-$. We have $\Pi^\pm_{c_0}(s_0)=\eta_\pm$, $(\Pi^-_c)'>0$, $(\Pi^+_c)'<0$ and $\Pi_c=\Pi^-_c\circ (\Pi^+_c)^{-1}$ where $s_0$ is defined in Section \ref{subsection-DM} and $\Pi_c$ is the Poincar\'{e} half-map defined in Section \ref{sec2.2}.
}
\end{remark}

\subsection{Proof of Theorem \ref{Thm-1}} \label{proof-1}

In this section, we consider $Z_{\epsilon,\tilde \lambda,c}$ and assume that $X_{0,c_0}^{sl}(\eta_{\pm})\ne 0$ and that the open interval $(\eta_{-},\eta_{+})$ contains singularities of $X_{0,c_0}^{sl}$. Statement 3 of Proposition \ref{prop-slow-sliding} implies that the slow dynamics of \eqref{family} for $c=c_0$ is regular at the corner points $x=\eta_{\pm}$ and has singularities in $(\eta_{-},\eta_{+})$ (away from $x=0)$. Following \cite[Section 4.2]{DD08}, we have that
\begin{align}\label{proofth2.2-1}
    \frac{\partial \widetilde{I}_{\pm}}{\partial x}(x)&=\frac{\partial I_\pm}{\partial x}(x,c)+o(1)\nonumber\\
    &=-\frac{(Y_{0,c}^{+}-Y_{0,c}^{-})^{2}(x,0)}{\det Z_{0,c}(x)}\phi'\left(\phi^{-1}\left(\frac{-Y_{0,c}^{-}}{Y_{0,c}^{+}-Y_{0,c}^{-}}(x,0)\right)\right)+o(1),
\end{align}
where $x$ is kept near $\eta_{\pm}$, $I_\pm$ are the slow divergence integrals defined in \eqref{slowdivint-attracting} and \eqref{slowdivint-repelling},  $\widetilde{I}_{\pm}$ are the integrals defined in Sections \ref{subsection-Del23} and \ref{subsection-combi} and $o(1)$ is $\epsilon-$regularly $C^{k}$ in $(x,\tilde{\lambda},c)$ tending to $0$ as $\epsilon\to 0$. The interval $(\eta_{-},\eta_{+})$ contains singularities of the slow dynamics for $c=c_0$, but note that $\frac{\partial I_\pm}{\partial x}$ are well-defined smooth functions in a small neighborhood of $(x,c)=(\eta_{\pm},c_0)$ because $X_{0,c_0}^{sl}(\eta_{\pm})\ne 0$ (see also \cite{DD08}).

From \eqref{proofth2.2-1} and Remark \ref{remark-bitno-new2} it follows that
\begin{equation}
    \label{proofth2.2-2}
    \frac{\partial}{\partial y}\left(\widetilde{I}_{-}(\Delta_{02}^-(y))-\widetilde{I}_{+}(\Delta_{02}^+(y))\right)=\frac{\partial I}{\partial x}(\Pi^+_c(y),c)(\Pi^+_c)'(y)+o(1),
\end{equation}
where $\frac{\partial I}{\partial x}$ is defined in \eqref{eq-der-SDI} and $o(1)$ has a similar property as above.

Suppose that $x=\eta_{+}$ is a zero of multiplicity $m$ of $\frac{\partial I}{\partial x}(x,c_0)$. Then $y=s_0$ is also a zero of multiplicity $m$ of $\hat I(y):=\frac{\partial I}{\partial x}(\Pi^+_{c_0}(y),c_0) (\Pi^+_{c_0})'(y)$, because $\Pi^+_{c_0}(s_0)=\eta_+$ and $(\Pi^+_c)'\ne 0$. This implies that
$$\frac{d^m\hat I}{d y^m}(s_0)\ne 0.$$

Using \eqref{proofth2.2-2} we conclude that
$$\frac{\partial^{m+1}}{\partial y^{m+1}}\left(\widetilde{I}_{-}(\Delta_{02}^-(y))-\widetilde{I}_{+}(\Delta_{02}^+(y))\right)\ne 0$$
near $y=s_0$ and $(\epsilon,\tilde\lambda,c)=(0,0,c_0)$. Thus, \eqref{equ-most-import} has at most $m+1$ solutions (counting multiplicity) w.r.t. $y$ near $y=s_0$ and $(\epsilon,\tilde\lambda,c)=(0,0,c_0)$. This shows that the sliding cycle $\Gamma$ can produce at most $m+2$ limit cycles. We have proved Theorem \ref{Thm-1}.

\subsection{Proof of Theorem \ref{Thm-2}} \label{proof-2}
Recall that $X_{0,c_0}^{sl}(0)>0$.
When $X_{0,c_0}^{sl}$ is regular in the stable sliding interval $[\eta_{-},0)$ (that is, the slow dynamics of \eqref{family} for $c=c_0$ is regular in $[\eta_{-},0)$), then it is shown (see \cite{DD08} and \cite{HK23}) that
\begin{equation}
    \label{SD-regular}
    \widetilde{I}_{-}(x)=I_-(x,c)+f_1(x,\epsilon,\tilde{\lambda},c)+f_2(\epsilon, \tilde{\lambda},c)\epsilon^2\ln\epsilon,
\end{equation}
where $f_1$ and $f_2$ are smooth functions, with $f_1(x,0,\tilde{\lambda},c)=0$, and $I_-$ is the slow divergence integral defined in \eqref{slowdivint-attracting}. From \eqref{SD-regular} it follows that $\widetilde{I}_{-}$ is bounded for $(x,\epsilon,\tilde{\lambda},c)$ kept near $(\eta_-,0,0,c_0)$. Since $\Delta_{02}^-(y)\to \eta_-$ as $(y,\epsilon,\tilde{\lambda},c)\to(s_0,0,0,c_0)$, we conclude that $\widetilde{I}_{-}(\Delta_{02}^-(y))$ is bounded near $(y,\epsilon,\tilde{\lambda},c)=(s_0,0,0,c_0)$. A similar result is true for $\widetilde{I}_{+}(\Delta_{02}^+(y))$ when $X_{0,c_0}^{sl}$ is regular in the unstable sliding interval $(0,\eta_{+}]$.

If $X_{0,c_0}^{sl}(\eta_{-})=0$ ($X_{0,c_0}^{sl}$ may have singularities in the stable sliding interval $(\eta_{-},0)$), then $$\widetilde{I}_{-}(x)\to -\infty, \ \left|\frac{\partial \widetilde{I}_{-}}{\partial x}(x)\right|\to +\infty$$
as $(x,\epsilon,\tilde{\lambda},c)\to(\eta_-,0,0,c_0)$ (see \cite{DD08}). Since $\Delta_{02}^-(y)\to \eta_-$ as $(y,\epsilon,\tilde{\lambda},c)\to(s_0,0,0,c_0)$ and $\frac{\partial \Delta_{02}^-}{\partial y}(y)\ne 0$, we conclude that $$\widetilde{I}_{-}(\Delta_{02}^-(y))\to -\infty, \ \left|\frac{\partial} {\partial y}\left(\widetilde{I}_{-}(\Delta_{02}^-(y))\right)\right|\to +\infty$$
as $(y,\epsilon,\tilde{\lambda},c)\to(s_0,0,0,c_0)$.
We have a similar result for $\widetilde{I}_{+}(\Delta_{02}^+(y))$ when $X_{0,c_0}^{sl}(\eta_{+})=0$.

\textit{Proof of Statement 1 of Theorem \ref{Thm-2}.}
 Assume that $x=\eta_{-}$ is a singularity of the sliding vector field $X_{0,c_0}^{sl}$ and that $X_{0,c_0}^{sl}$ is regular in the unstable sliding interval $(0,\eta_{+}]$. Then the above discussion implies that $\widetilde{I}_{-}(\Delta_{02}^-(y))\to -\infty$
as $(y,\epsilon,\tilde{\lambda},c)\to(s_0,0,0,c_0)$ and $\widetilde{I}_{+}(\Delta_{02}^+(y))$ is bounded.
 Thus, the equation in \eqref{equ-most-import} has no solutions and $\Gamma$ produces at most $1$ limit cycle. Since the total divergence integral $\widetilde{I}_{-}(\Delta_{02}^{-}(y)) - \widetilde{I}_{+}(\Delta_{02}^{+}(y))$ is negative, the limit cycle, if it exists, is hyperbolic and attracting.

\textit{Proof of Statement 2 of Theorem \ref{Thm-2}.}  Analogous to the proof of Statement 1.

\textit{Proof of Statement 3 of Theorem \ref{Thm-2}.} We assume that $X_{0,c_0}^{sl}(\eta_{-})=0$ and $X_{0,c_0}^{sl}(\eta_{+})\neq 0$ ($X_{0,c_0}^{sl}$ may have singularities in $(\eta_{-},\eta_{+})$). Then $\left|\frac{\partial} {\partial y}\left(\widetilde{I}_{-}(\Delta_{02}^-(y))\right)\right|\to +\infty$
as $(y,\epsilon,\tilde{\lambda},c)\to(s_0,0,0,c_0)$, and $\frac{\partial} {\partial y}\left(\widetilde{I}_{+}(\Delta_{02}^+(y))\right)$ is bounded (see Section \ref{proof-1}). This implies that
$$\frac{\partial}{\partial y}\left(\widetilde{I}_{-}(\Delta_{02}^-(y))-\widetilde{I}_{+}(\Delta_{02}^+(y))\right)$$ is nonzero near $(s_0,0,0,c_0)$. Thus, the equation in \eqref{equ-most-import} has at most $1$ solution and $\Gamma$ generates at most $2$ limit cycles.

We use a similar proof if $X_{0,c_0}^{sl}(\eta_{-})\ne 0$, $X_{0,c_0}^{sl}(\eta_{+})=0$.

\subsection{Proof of Theorem \ref{Thm-3}} \label{proof-3}
The proof of Theorem \ref{Thm-3} is similar to the proof of \cite[Theorem 2.26]{DD08}. However, for the sake of completeness, we will prove it. We assume that $x=\eta_\pm$ is a zero of multiplicity $m_\pm\ge 1$ of $X_{0,c_0}^{sl}$, with $m_-\ne m_+$ ($X_{0,c_0}^{sl}$ may have singularities in $(\eta_{-},\eta_{+})$). Using \cite{DD08} (see also \cite[Section 3.1.2]{H22}) we have
\begin{equation}
    \label{eq-m+-}
    \frac{\partial \widetilde{I}_{\pm}}{\partial x}(x)=\frac{\left(1+O(\epsilon^2)\right)\frac{\partial\beta_\pm}{\partial x}(x)}{h_\pm(\beta_\pm(x))}+O(1),
\end{equation}
where $h_\pm$ and $\beta_\pm$ are defined in Remark \ref{remark-nor-for}, $O(\epsilon^2)$ is a $C^k$-function and $O(1)$ is $\epsilon-$regularly $C^{k}$ in $(x,\tilde{\lambda},c)$.
\begin{remark}\label{remark-nor-for}{\rm
    The first term in \eqref{eq-m+-} represents the derivative of a portion of the divergence integral $\widetilde{I}_{\pm}$ calculated near the corner point $x=\eta_\pm$ in a $C^k$-normal form for equivalence
    \begin{equation}\label{Cknormalform}
\begin{array}{rcl}
\dot{v}_{1} & = & \mp\epsilon^2h_\pm(v_1,\epsilon,\tilde\lambda,c), \\[3mm]
\dot{v}_{2} & = & \pm v_2,
\end{array}
\end{equation}
between $\{v_2=1\}$, parameterized by $v_1\sim 0$, and $\{v_1=1\}$, parameterized by $v_2$ (see \cite[Section 4.5]{DD08}). This is calculated in negative time near $x=\eta_+$.  We suppose that the orbit $\mathcal{O_{\pm}}(x,\epsilon,\tilde\lambda,c)$, defined above, intersects the section $\{v_2=1\}$ in $v_1=\beta_\pm(x,\epsilon,\tilde\lambda,c)$. The function $\beta_-$ (resp. $\beta_+$) is a $C^k$-diffeomorphism with $\beta_-(\eta_-,0,0,c_0)=0$ and $\frac{\partial\beta_-}{\partial x}>0$ (resp. $\beta_+(\eta_+,0,0,c_0)=0$ and $\frac{\partial\beta_+}{\partial x}<0$), due to the chosen parametrization of $S_2$, $S_2'$ and $\{v_2=1\}$. In normal form coordinates \eqref{Cknormalform}, the slow dynamics is given by $v_1'=\mp h_\pm(v_1,0,0,c)$. We have that $v_1=0$ is a zero of multiplicity $m_\pm$ of $h_\pm(v_1,0,0,c_0)$.

    }
\end{remark}

Denote $H^{\pm}(y) = \beta_{\pm}\circ\Delta_{02}^{\pm}(y)$ and $\frac{\partial H^{\pm}}{\partial y}(y) = \frac{\partial\beta_\pm}{\partial x}(\Delta_{02}^\pm(y))\frac{\partial \Delta_{02}^\pm}{\partial y}(y)$. We have
 \begin{equation*}
     \frac{\partial}{\partial y}\left(\widetilde{I}_{-}(\Delta_{02}^-(y))-\widetilde{I}_{+}(\Delta_{02}^+(y))\right)\nonumber  = \frac{\left(1+O(\epsilon^2)\right)\frac{\partial H^{-}}{\partial y}(y)}{h_{-}(H^{-}(y))} - \frac{\left(1+O(\epsilon^2)\right)\frac{\partial H^{+}}{\partial y}(y)}{h_{+}(H^{+}(y))} + O(1).
 \end{equation*}
 
Now we can write
\begin{align}\label{eq-long}
 h_{-}(H^{-}(y)) h_+(H^{+}(y))  & \frac{\partial}{\partial y}\left(\widetilde{I}_{-}(\Delta_{02}^-(y))-\widetilde{I}_{+}(\Delta_{02}^+(y))\right)\nonumber = \\
  = & \left(1+O(\epsilon^2)\right)\frac{\partial H^{-}}{\partial y}(y)h_+(H^{+}(y))\nonumber \\
  & - \left(1+O(\epsilon^2)\right)\frac{\partial H^{+}}{\partial y}(y)h_{-}(H^{-}(y)) + h_{-}(H^{-}(y)) h_{+}(H^{+}(y)) O(1).
 \end{align}
 
Since the partial derivatives of $\beta_\pm$ (resp. $\Delta_{02}^\pm$) w.r.t. $x$ (resp. $y$) are nonzero, $H^{\pm}(s_{0}) = \beta_\pm(\Delta_{02}^\pm(s_0))=0$ for $(\epsilon,\tilde{\lambda},c)=(0,0,c_0)$, and $h_\pm$ has a zero of multiplicity $m_\pm$ ($m_-\ne m_+$) at the origin for $(\epsilon,\tilde{\lambda},c)=(0,0,c_0)$, we conclude that the right-hand side of \eqref{eq-long} for
 $(\epsilon,\tilde{\lambda},c)=(0,0,c_0)$ has a zero of multiplicity $\min\{m_-,m_+\}$ at $y=s_0$. Thus, the right-hand side of \eqref{eq-long} has at most $\min\{m_-,m_+\}$ zeros (counting multiplicity) near $(s_0,0,0,c_0)$. This implies that $\widetilde{I}_{-}(\Delta_{02}^-(y))-\widetilde{I}_{+}(\Delta_{02}^+(y))$ has at most $1+\min\{m_-,m_+\}$ zeros and $\Gamma$ can therefore produce at most $2+\min\{m_-,m_+\}$ limit cycles.

\subsection{Proof of Theorem \ref{Thm-4}} \label{proof-4}
In this section we assume that $X_{0,c_0}^{sl}$ (or, equivalently, the slow dynamics of \eqref{family}) has a simple zero at both corner points.
Using the reduction of the dynamics of \eqref{family} to invariant manifolds (see the proof of Proposition \ref{prop-slow-sliding}) and the implicit function theorem, it can be easily seen that the simple zero $x=\eta_\pm$ of $X_{0,c_0}^{sl}$ generates a hyperbolic saddle of \eqref{family} for $\epsilon>0$ with ratio of eigenvalues  $-\epsilon^{2}\rho_{\pm}$, $\rho_{\pm}=\rho_{\pm}(\epsilon,\tilde \lambda,c)>0$. This implies that \eqref{PWSRescaled} has a hyperbolic saddle near $x=\eta_\pm$ with the same ratio of eigenvalues.

Assume that $\rho_{-}\ne \rho_{+}$ for $(\epsilon,\tilde \lambda,c)=(0,0,c_0)$.

\textit{Proof of Statement 1 of Theorem \ref{Thm-4}.} We assume that the interval $(\eta_{-},\eta_{+})$ does not contain singularities of $X_{0,c_0}^{sl}$. From \cite[Section 4.4]{DD08} it follows that
\begin{equation}
    \label{eq-proof-ln}
    \widetilde{I}_{\pm}(x)=\frac{1}{\rho_\pm}\ln\beta_\pm(x,\epsilon,\tilde\lambda,c)+O(1),
\end{equation}
where $\beta_\pm$ and $O(1)$ have the same properties like $\beta_\pm$ and $O(1)$ in \eqref{eq-m+-}.
\begin{remark}{\rm
    The $\log$-term in \eqref{eq-proof-ln} is a portion of the divergence integral $\widetilde{I}_{\pm}$ calculated near the simple zero $x=\eta_\pm$ in a $C^k$-normal form for equivalence $\{\dot v_1=\mp\epsilon^2\rho_\pm v_1,\dot v_2=\pm v_2\}$ from the section $\{v_2=1\}$ to the section $\{v_1=1\}$. We refer to \cite[Section 4.4]{DD08}. Again, it is calculated in negative time near $x=\eta_+$.
    }
\end{remark}

Once again we adopt the notation $H^{\pm}(y) = \beta_{\pm}\circ\Delta_{02}^{\pm}(y)$. It follows that
$$\widetilde{I}_{-}(\Delta_{02}^-(y))-\widetilde{I}_{+}(\Delta_{02}^+(y))=\frac{1}{\rho_-}\ln H^{-}(y) - \frac{1}{\rho_+}\ln H^{+}(y) + O(1).$$

If we use the change of variable $v = H^{-}(y) = \beta_-(\Delta_{02}^-(y))$ (note that $\frac{\partial H^{-}}{\partial y}(y) \neq 0$), then the above expression becomes
$$\frac{1}{\rho_-}\ln v-\frac{1}{\rho_+}\ln\tilde\beta_+(v)+O(1).$$

The orbit through $v\in\{v_2=1\}$ near $x=\eta_-$ intersects the section $\{v_2=1\}$ near $x=\eta_+$ (in negative time) at a point with $v_1=\tilde\beta_+(v)$.
Observe that $\tilde\beta_+(0)=0$ for $(\epsilon,\tilde{\lambda},c)=(0,0,c_0)$ and  $\frac{\partial \tilde\beta_+}{\partial v}\ne 0$. Following \cite[Lemma 4.11]{DD08}, this expression has at most $1$ zero near $(v,\epsilon,\tilde{\lambda},c)=(0,0,0,c_0)$, which is simple. This implies Statement 1.

\textit{Proof of Statement 2 of Theorem \ref{Thm-4}.} We assume that the interval $(\eta_{-},\eta_{+})$ contains singularities of $X_{0,c_0}^{sl}$. We will use the idea explained in the proof of Theorem \ref{Thm-3}. Since $x=\eta_\pm$ are simple zeros of $X_{0,c_0}^{sl}$, we have $m_\pm=1$. Note that we can rescale the $v_1$-variable in \eqref{Cknormalform} so that $$\frac{\partial H^{\pm}}{\partial y}(y) = \frac{\partial }{\partial y}\beta_\pm(\Delta_{02}^\pm(y))=1$$ for $(y,\epsilon,\tilde{\lambda},c)=(s_0,0,0,c_0)$. Then the right-hand side of \eqref{eq-long} can be written as
$$(\rho_+-\rho_-)(y-s_0)+ O\left((y-s_0)^2\right)$$
for $(\epsilon,\tilde{\lambda},c)=(0,0,c_0)$. This expression has $1$ zero (multiplicity counted) near $y=s_0$ because $\rho_-\ne\rho_+$. This implies Statement 2.

\section{Applications to regularized piecewise linear
visible-invisible two-folds}
\label{sec4}
In this section, we study limit cycles of regularized
piecewise linear (PWL) visible-invisible two-fold singularities \cite{HK24}. We focus on sliding cycles $\Gamma$ that contain a singularity of the sliding vector field as their corner point (the cyclicity of sliding cycles with a regular sliding vector field has been studied in \cite{HK24} by using the slow divergence integral \eqref{SDI-statement}).

Consider a planar PWL system
\begin{equation}\label{eq-pwl-general}
Z^{\pm}(z)=A^{\pm}z+b^{\pm},
\end{equation}
with $z=(x,y)$, $A=(A_{ij}^{\pm})$, $b=(b_{1}^{\pm},b_{2}^{\pm})^{T}$ and the switching manifold $\{y=0\}$. If system \eqref{eq-pwl-general} satisfies the first line of \eqref{A1-1} and $\nu(c)>0$, with $\nu(c)$ defined in Section \ref{subsection-assump}, then there exists an affine transformation bringing \eqref{eq-pwl-general} into
\begin{equation}\label{PWL}
\tilde{Z}^{-}({x},{y})=\left(
\begin{array}{c}
-1+d^{-}{y} \\[3mm]
-{x}+t^{-}{y}
\end{array}
\right),
\quad \quad  \tilde{Z}^{+}({x},{y})=\left(
\begin{array}{c}
 b^{+}+a_{11}^{+}{x}+a_{12}^{+}{y} \\[3mm]
a_{21}^{+}{x}+a_{22}^{+}{y}
\end{array}
\right),
\end{equation}
with $b^{+}>a_{21}^{+}>0$, $t^{-}=\text{tr}(DZ^{-})$  and $d^{-}=\det(DZ^{-})$. For more details, we refer to \cite[Proposition 2.1]{HK24}. It can be checked that $\Sigma_{sl}=(-\infty,0)\cup(0,\infty)$ and $\Sigma_{T}=\{0\}$.

By \eqref{slvf}, the Filippov sliding vector field is
\begin{equation}
X^{sl}({x})=\frac{1}{1+a_{21}^{+}}(b^{+}-a_{21}^{+}+a_{11}^{+}{x}).\nonumber
\end{equation}
The sliding vector field $X^{sl}$ has a singularity $ x^{*}=\frac{a_{21}^{+}-b^{+}}{a_{11}^{+}}\ne 0$ of multiplicity $1$ when $a_{11}^{+}\neq0$. Observe that $x^{*} < 0$ (resp. $x^{*} > 0$) if, and only if, $a_{11}^{+} > 0$ (resp. $a_{11}^{+} < 0$).

We study limit cycles of the following regularized planar PWL system (see \eqref{PWSRescaled})
\begin{equation}\label{PWSRescaledLinear}
\dot{z}=Z_{\epsilon\tilde\lambda,c}^{+}(z)\phi(y\epsilon^{-2})+Z^{-}_{\epsilon\tilde\lambda,c}(z)(1-\phi(y\epsilon^{-2})),
\end{equation}
where $0<\epsilon\ll1$, $\tilde \lambda\sim 0$, $Z_{\epsilon\tilde\lambda,c}^\pm$ are smooth planar affine vector fields, $Z_{0,c}^\pm=\tilde{Z}^{\pm}$ and $\phi$ satisfies (A4). We take $c=(d^{-},t^{-},b^{+},a_{11}^{+},a_{12}^{+},
a_{21}^{+},a_{22}^{+})$.

Since the sliding cycles are located in $\{y\leq0\}$, we first have to investigate the dynamics of $\tilde{Z}^{-}$. In what follows, we briefly describe the phase portrait of $\tilde{Z}^{-}$ based on the coefficients $d^{-}$ and $t^{-}$. The following facts are relevant when we study such sliding cycles (see also \cite[Section 2.4]{HK24}).

If $d^{-}=0$, then $\tilde{Z}^{-}$ has no singularities. In particular, if $t^{-}\neq0$, then $\tilde{Z}^{-}$ has an invariant line $x = t^{-}y+\frac{1}{t^{-}}$ and, if $t^{-}=0$, then orbits of $\tilde{Z}^{-}$ are parabolas $y=\frac{1}{2}x^2+C$. Now, suppose $d^{-} \neq 0$. Then $\tilde{Z}^{-}$ has a singularity positioned at $P = (\frac{t^{-}}{d^{-}},\frac{1}{d^{-}})$, with eigenvalues $$\kappa_{\pm}=\frac{t^{-}\pm\sqrt{(t^{-})^{2}-4d^{-}}}{2}.$$
We write $x_L,x_R=\frac{1}{\kappa_\pm}$ where we suppose that $x_L<x_R$ if $x_L\ne x_R$. Now, we have

\begin{itemize}
    \item[\textbf{(1)}] If $d^{-}<0$, then $P$ is a hyperbolic saddle whose stable and unstable straight manifolds intersect the switching manifold at $x={x}_{R}>0$ and $x={x}_{L}<0$, respectively.
\end{itemize}
If $d^{-}>0$, then $P$ is a virtual singularity. In this case, it also follows that

\begin{itemize}
    \item[\textbf{(2)}] If $(t^{-})^2-4d^{-} \geq 0$, then $P$ is a (virtual) node whose separatrices intersect the switching manifold at points $(0, x_{L})$ and $(0, x_{R})$, where $0<x_L\le x_R$ or $x_L\le x_R<0$. Moreover, $P$ is repelling when $t^{-} > 0$  and attracting when $t^{-} < 0$.
    \item[\textbf{(3)}] If $(t^{-})^2-4d^{-} < 0$, then $P$ is a (virtual) monodromic point. It is a repelling focus if $t^{-} > 0$, an attracting one if $t^{-} < 0$ and a center if $t^{-} = 0$.
    \end{itemize}

Now we are able to study the number of limit cycles of \eqref{PWSRescaledLinear} produced by sliding cycles of \eqref{PWL}, with the simple zero $x^*$ of the sliding vector field at a corner point.  The next result follows directly from Statements 1 and 2 of Theorem \ref{Thm-2}.

\begin{theorem}\label{thm-appl}
Suppose that $a_{11}^{+}\neq 0$ in system \eqref{PWL}. Assume that \eqref{PWL} satisfies one of the following configurations.
\begin{enumerate}
    \item(\textbf{The case without singularities}) Suppose that $d^{-}=0$, and assume that one of the following is satisfied \textbf{(a)} $t^{-}=0$; or \textbf{(b)} $t^{-}>0$ and $ x^{*}<\frac{1}{t^{-}}$; or \textbf{(c)} $t^{-}<0$ and $ x^{*}>\frac{1}{t^{-}}$. See Cases I, II and III in Figure \ref{fig-piecewise}.
    \item(\textbf{Saddle case}) Suppose that $d^{-}<0$, and $\ {x}_L <  {x}^{*} < {x}_R$. See Cases IV and V in Figure \ref{fig-piecewise}.
    \item \textbf{(Node case)} Suppose that $d^{-}>0$ and $(t^{-})^2-4d^{-} \geq 0$, and assume that one of the following is satisfied \textbf{(a)} $t^{-}>0$ and ${x}^{*}<{x}_L, {x}_R$; or \textbf{(b)} $t^{-} < 0$ and ${x}^{*} > {x}_L, {x}_R$. See Cases VI and VII in Figure \ref{fig-piecewise}.
    \item \textbf{(Monodromic case)} Suppose that $d^{-}>0$ and $(t^{-})^2-4d^{-}<0$. See Cases VIII, IX and X in Figure \ref{fig-piecewise}.
\end{enumerate}

Then the cyclicity of sliding cycles having the singularity $x^*$ at the corner point is bounded by $1$. Moreover, if it exists, the limit cycle is hyperbolic and attracting (resp. repelling) if $a_{11}^{+}  >0$ (resp. $a_{11}^{+} < 0$).
\end{theorem}


\begin{figure}[ht]\center{
\begin{overpic}[width=1\textwidth]{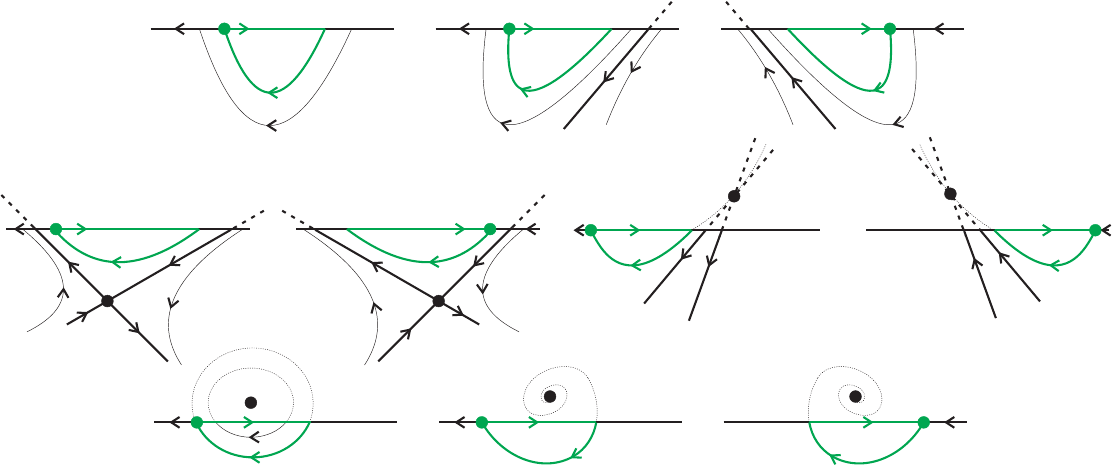}

\put(24,28){I}
\put(45,28){II}
\put(76,28){III}

\put(8,10){IV}
\put(39,10){V}
\put(65,11){VI}
\put(85,11){VII}

\put(22,-2){VIII}
\put(50,-2){IX}
\put(76,-2){X}

\end{overpic}}
\caption{Cases in Theorem \ref{thm-appl}. Observe that in cases I, II, IV, VI, VII and IX the singularity $x^{*}$ is positioned at $\eta_{-}$, however the case in which $x^{*}$ is positioned at $\eta_{+}$ is also allowed. Similarly, In cases III, V, VII and X the singularity $x^{*}$ is positioned at $\eta_{+}$, but it could be positioned at $\eta_{-}$ as well.}\label{fig-piecewise}
\end{figure}
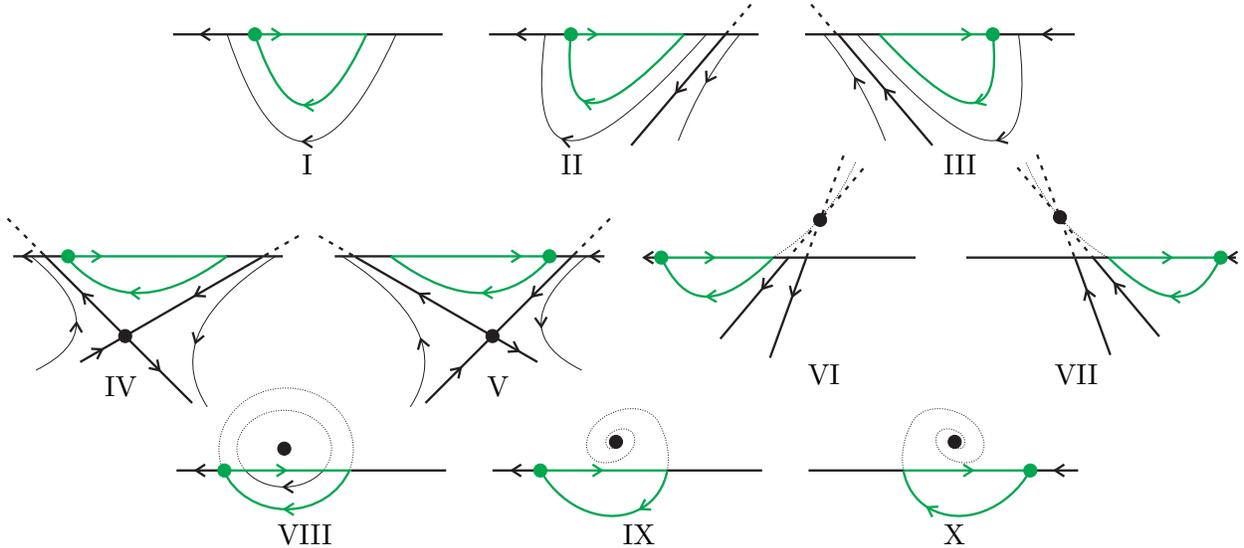

\appendix

\section{Blowing up the degenerate line $\overline{H}$}\label{appendix-extra-blowup}

In this section, we show that the slow-fast system \eqref{family} is in the framework of \cite{DD08}, near the intersection $p_0$ of $\overline{C}$ with $\overline{H}$ (see Section \ref{subsec-e1}). Recall that system \eqref{family} satisfies the assumptions of \cite{DD08} along $[\eta_{-}, \eta_{+}]\subset\overline{C}$, but away from $p_0$ (see Proposition \ref{prop-slow-sliding} and Remark \ref{remark-bitno}).



For this purpose, we must blow up the degenerate line $\overline{H}$ (see Figure \ref{FIG6}). We apply the (cylindrical) singular transformation
\begin{equation*}
(\hat{x},y,\hat{\epsilon},\rho)  \mapsto  (\rho\hat{x}, y, \rho\hat{\epsilon}) = (x_{2},y_{2},r_{2}),
\end{equation*}
where $(\hat x,\hat \epsilon)\in\mathbb S^1$, $\hat \epsilon\ge 0$ and $\rho\ge 0$. As usual, we will work with different charts near the blow-up locus $\rho=0$ (the phase directional charts $\hat{x} = \pm 1$ and the family chart $\hat{\epsilon} = 1$).

\begin{figure}[htp]
\begin{center}
\includegraphics[width=4in]{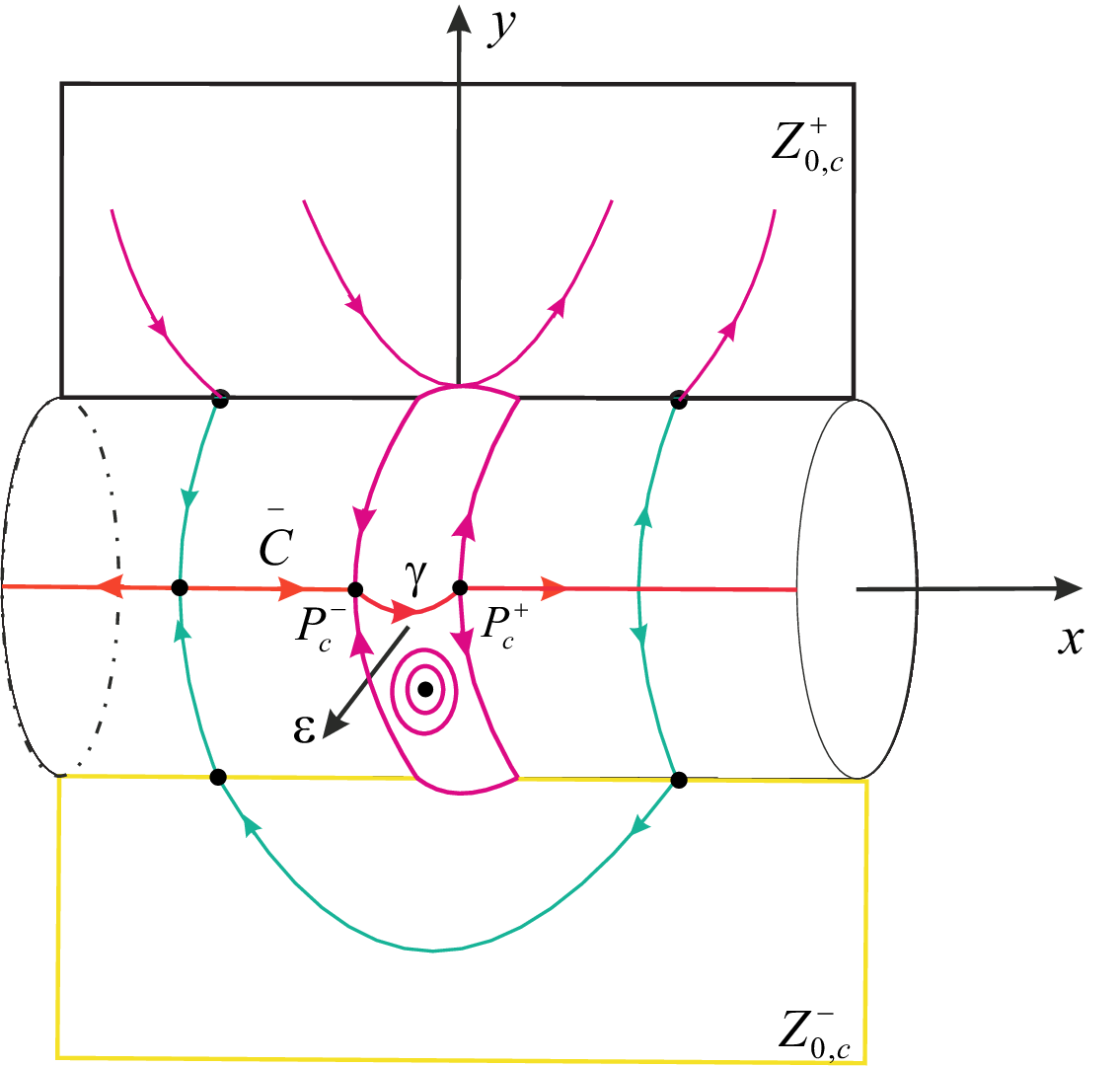}
\caption{\footnotesize{Blow-up of the degenerate line $\overline{H}$}}\label{FIG6}
\end{center}
\end{figure}

\subsection{Charts $\hat{x} = \pm 1$}

Let us first study the chart $\hat{x} = - 1$. The blow-up transformation takes the from $(x_{2}, r_{2}) = (-\rho_{1}, \rho_{1}\hat{\epsilon}_{1})$ and system \eqref{family} becomes (after division by $\rho_{1} > 0$)
\begin{equation}\label{eq-app-chart-x1-after}
\begin{array}{rcccl}
\dot{\rho}_1 & = & -\rho_{1}\hat{\epsilon}_{1}^{2}\left(X^{+}_{0, c}\left(0,0\right)\phi(y_{2}) + X^{-}_{0, c}\left(0,0\right)(1-\phi(y_{2}))+O(\rho_{1})\right), \\[3mm]
\dot{y}_2 & = & -\left(\partial_{x}Y^{+}_{0, c}\left(0,0\right)\phi(y_{2}) + \partial_{x}Y^{-}_{0, c}\left(0,0\right)(1-\phi(y_{2}))\right) & + & O(\rho_{1},\hat{\epsilon}_{1}),  \\[3mm]
\dot{\hat{\epsilon}}_1 & = & \hat{\epsilon}_{1}^{3}\left(X^{+}_{0, c}\left(0,0\right)\phi(y_{2}) + X^{-}_{0, c}\left(0,0\right)(1-\phi(y_{2}))+O(\rho_{1})\right),
\end{array}
\end{equation}
where we use $Y_{0,c}^\pm(0,0)=0$ (see \eqref{A1-1}).
In the invariant line $\{\rho_{1} = \hat{\epsilon}_{1} = 0\}$, system \eqref{eq-app-chart-x1-after} has a normally hyperbolic singularity $P_{c}^{-} = (0,y_{2}^{*},0)$, with
$$y_{2}^{*} := \phi^{-1}\left(\frac{-\partial_{x} Y_{0,c}^{-}}{\partial_{x}Y_{0,c}^{+}-\partial_{x}Y_{0,c}^{-}}(0,0)\right).$$

The nonzero eigenvalue of $P_{c}^{-}$ is given by $- \phi'(y_{2}^{*})\left(\partial_{x}Y_{0,c}^{+}-\partial_{x}Y_{0,c}^{-}\right)(0,0) < 0$ (we use \eqref{A1-1} and $\phi'>0$), and center manifolds at $P_{c}^{-}$ are transverse to the stable manifold given by the $y_2$-axis. Observe that $P_{c}^{-}$ is the end point of the normally attracting portion of $\overline{C}$ located on the blow-up locus (see Figure \ref{FIG6}). This holds for every $(\tilde{\lambda},c)$ in a small neighborhood of $(0,c_0)$, and therefore Assumption $2.10$ of \cite{DD08} is fulfilled.

Now we check Assumption $2.11$. A center manifold of \eqref{eq-app-chart-x1-after} at $P_{c}^{-}$ is a graph in $(\rho_{1},\hat{\epsilon}_{1})$
$$y_{2} = y_{2}^{*} + O(\rho_{1},\hat{\epsilon}_{1}).$$

Replacing this last expression in the $(\rho_{1},\hat{\epsilon}_{1})$ component of \eqref{eq-app-chart-x1-after} yields to the center behavior
\begin{equation}\label{eq-app-center-bhv}
\dot{\rho}_{1} = -\rho_{1}\hat{\epsilon}_{1}^{2}\left(X^{sl}_{0,c}(0) + O(\rho_{1},\hat{\epsilon}_{1})\right), \quad \dot{\hat{\epsilon}}_{1} = \hat{\epsilon}_{1}^{3}\left(X^{sl}_{0,c}(0) + O(\rho_{1},\hat{\epsilon}_{1})\right).
\end{equation}

Recall from Assumption (A2) that $X^{sl}_{0,c}(0)$ is positive. Now, we divide \eqref{eq-app-center-bhv} by $\hat{\epsilon}_{1}^{2}$ and find an isolated hyperbolic
saddle at $(\rho_{1},\hat{\epsilon}_{1})=(0,0)$. This is true for all $(\tilde{\lambda},c)$ in a small neighborhood of $(0,c_0)$. Therefore, Assumption $2.11$ is checked. Note also that $P_{c}^{-}$ has a unique center manifold on the blow-up locus $\rho_1=0$ because \eqref{eq-app-chart-x1-after} is of saddle type inside $\rho_1=0$.

\begin{remark}
{\rm
The exponent in $\hat{\epsilon}_{1}^{2}$, sometimes called the order of degeneracy, is equal to the exponent in the term $r_{2}^{2}$ that we divide out when we define the slow dynamics of \eqref{family} along $\overline{C}$ (see the proof of Proposition \ref{prop-slow-sliding}).
}
\end{remark}

So far, we checked assumptions $2.10$ and $2.11$ in the chart $\hat{x} = -1$. The study of the chart $\hat{x} = 1$ is completely analogous (we can apply $(\rho_{1},\hat{\epsilon}_{1},t)\to(-\rho_{1},-\hat{\epsilon}_{1},-t)$ to \eqref{eq-app-chart-x1-after}). We denote the intersection between the repelling portion of $\overline{C}$ and the blow-up locus by $P_{c}^{+}$.

\subsection{Chart $\hat{\epsilon}$ = 1}

In this chart we have $x_{2} = \rho_{2}\hat{x}_{2}$ and $r_2=\rho_{2}$. System \eqref{family} changes into
\begin{equation}\label{eq-app-chart-e1}
\begin{array}{rcl}
\dot{\hat{x}}_2 & = & \rho_{2}\left(X^{+}_{\rho_2\tilde \lambda, c}(\rho_{2}\hat{x}_{2},\rho_{2}^{2}y_2)\phi(y_2) + X^{-}_{\rho_2\tilde \lambda, c}(\rho_{2}\hat{x}_{2},\rho_{2}^{2}y_2)\left(1-\phi(y_2)\right)\right), \\[3mm]
\dot{y}_2 & = & Y^{+}_{\rho_2\tilde\lambda, c}(\rho_{2}\hat{x}_{2},\rho_{2}^{2}y_2)\phi(y_2)+Y^{-}_{\rho_2\tilde\lambda, c}(\rho_{2}\hat{x}_{2},\rho_{2}^{2}y_2)\left(1-\phi(y_2)\right).
\end{array}
\end{equation}

In order to study the dynamics on the blow-up locus, we first divide the right-hand side of \eqref{eq-app-chart-e1} by $\rho_{2}$ and then set $\rho_{2} = 0$ and $\tilde \lambda=0$. Thus we obtain
\begin{equation}\label{eq-app-chart-e1-after}
\begin{array}{rcl}
\dot{\hat{x}}_2 & = & X^{+}_{0, c}(0,0)\phi(y_2) + X^{-}_{0, c}(0,0)\left(1-\phi(y_2)\right), \\[3mm]
\dot{y}_2 & = & \hat{x}_{2}\left(\partial_{x}Y^{+}_{0, c}(0,0)\phi(y_2)+\partial_{x}Y^{-}_{0, c}(0,0)\left(1-\phi(y_2)\right)\right).
\end{array}
\end{equation}

It can be easily checked that for each $c\sim c_0$, the horizontal line $\gamma = \{y_{2} = y_{2}^{*}\}$ is invariant by \eqref{eq-app-chart-e1-after}. The line $\gamma$ is regular (we have $\dot{\hat{x}}_2 = X^{sl}_{0,c}(0)>0$ on it) and connects $P_{c}^{-}$ to $P_{c}^{+}$. We refer to Figure \ref{FIG6}.  Therefore, Assumption $2.13$ of \cite{DD08} also holds.

\begin{remark}
{\rm
In \cite{HK23}, a planar slow-fast system similar to \eqref{family} satisfies the assumptions of \cite{DD05}. The main difference between our paper and \cite{HK23} lies in the fact that in \cite{HK23} the associated slow dynamics is regular and the invariant line $\gamma$ breaks in a regular way as we vary $\tilde\lambda$ (see the additional assumption in Remark \ref{rem-breaking}).
}
\end{remark}

\section*{Declarations}
\textbf{Conflict of interest} The authors declare that they have no conflict of interest.

\textbf{Data Availability Statement} Data sharing not applicable to this article as no datasets were generated or
analysed during the current study.


\begin{thebibliography}{50}
\bibitem{BLM18}
     \newblock C. Bonet-Reves, J. Larrosa and T. M-Seara,
     \newblock Regularization around a generic codimension one fold-fold singularity,
     \newblock \emph{J. Differential Equations}, \textbf{265} (2018), 1761--1838.

\bibitem{BST06}
    \newblock C. Buzzi, P.R. Silva, M.A. Teixeira.
    \newblock A singular approach to discontinuous vector fields on the plane,
    \newblock \emph{J. Differential Equations} \textbf{231} (2006), 633--655.

\bibitem{DD05}
     \newblock P. De Maesschalck and F. Dumortier,
     \newblock Time analysis and entry-exit relation near planar turning points,
     \newblock \emph{J. Differential Equations} \textbf{215} (2005), 225--267.
\bibitem{DD08}
     \newblock P. De Maesschalck and F. Dumortier,
     \newblock Canard cycles in the presence of slow dynamics with singularities,
     \newblock \emph{Proc. Roy. Soc. Edinburgh Sect. A} \textbf{138} (2008), 265--299.
\bibitem{DDR21}
     \newblock P. De Maesschalck, F. Dumortier and R. Roussarie,
     \newblock \emph{Canard Cycles from Birth to Transition},
     \newblock Springer, Cham, 2021.
\bibitem{dBCK08}
     \newblock M. di Bernardo, C. J. Budd, A.R. Champneys and P. Kowalczyk,
     \newblock Piecewise-smooth Dynamical systems: Theory and Applications,
     \newblock \emph{Springer Verlag}, (2008).
\bibitem{Fen79}
    \newblock N. Fenichel,
    \newblock Geometric singular perturbation theory for ordinary differential equations,
     \newblock \emph{J. Differential Equations}, \textbf{31} (1979), 53--98.
\bibitem{Filippov}
    \newblock A.F. Filippov,
    \newblock Differential Equations with Discontinuous Right-Hand Sides, Mathematics and Its Applications (Soviet Series),
    \newblock \emph{Kluwer Academic Publishers}, Dordrecht, 1988.

\bibitem{GSTJDE}
    \newblock M. Guardia, T. M. Seara and M. A. Teixeira,
    \newblock Generic bifurcations of low codimension
of planar Filippov systems,
    \newblock \emph{J. Differential Equations}, \textbf{250} (2011), No. 4, 1967--2023.


\bibitem{H22}
     \newblock R. Huzak,
     \newblock Cyclicity of canard cycles with hyperbolic saddles located away from the critical curve,
     \newblock \emph{J. Differential Equations} \textbf{320} (2022), 479--509.
\bibitem{HK23}
     \newblock R. Huzak and K. Uldall Kristiansen,
     \newblock The number of limit cycles for regularized piecewise polynomial systems is unbounded,
     \newblock \emph{J. Differential Equations}, \textbf{342} (2023), 34--62.
\bibitem{HK24}
     \newblock R. Huzak and K. Uldall Kristiansen,
     \newblock Sliding Cycles of Regularized Piecewise Linear Visible-Invisible Twofolds,
     \newblock \emph{Qual. Theory Dyn. Syst.}, \textbf{23} (2024), https://doi.org/10.1007/s12346-024-01111-y.
\bibitem{HKR}
     \newblock R. Huzak, K. Uldall Kristiansen and G. Radunovi\'c,
     \newblock Slow divergence integral in regularized piecewise smooth systems,
     \newblock \emph{Electron. J. Qual. Theory Differ. Equ.}, \textbf{86} (2024), 1--30.
\bibitem{JKW21a}
     \newblock S. Jelbart, K. Uldall Kristiansen and M. Wechselberger,
     \newblock Singularly perturbed boundary-focus bifurcations,
     \newblock \emph{J. Differential Equations}, \textbf{296} (2021), 412--492.
\bibitem{JKW21b}
     \newblock S. Jelbart, K. Uldall Kristiansen and M. Wechselberger,
     \newblock Singularly perturbed boundary-equilibrium bifurcations,
     \newblock \emph{Nonlinearity}, \textbf{34} (2021), 7371--7414.
\bibitem{KH15}
     \newblock K. Uldall Kristiansen and S. J. Hogan,
     \newblock Regularizations of two-fold bifurcations in planar piecewise smooth systems using blowup,
     \newblock \emph{SIAM J. Appl. Dyn. Sys.}, \textbf{14} (2015), 1731--1786.
\bibitem{K20}
     \newblock K. Uldall Kristiansen,
     \newblock The regularized visible fold revisted,
     \newblock \emph{J. Nonlinear Sci.}, \textbf{30} (2020), 2463--2511.
\bibitem{KRG03}
     \newblock Yu. A. Kuznetsov, S. Rinaldi and A. Gragnani,
     \newblock One parameter bifurcations in planar Filippov systems,
     \newblock \emph{Int. J. Bifur. Chaos}, \textbf{13} (2003), 2157--2188.
\bibitem{NoJef} \newblock D.D. Novaes and M.R. Jeffrey,
\newblock Regularization of hidden dynamics in piecewise smooth flows,
\newblock \emph{J. Differential Equations}, \textbf{259} (2015), 4615--4633.

\bibitem{LliSilTei06} \newblock J. Llibre,  P.R. da Silva and M.A. Teixeira,
\newblock Regularization of discontinuous vector fields on $\mathbb{R}^{3}$ via singular perturbation,
\newblock \emph{J. Dynamics and Differential Equations} \textbf{19} (2006), 309--331.


\bibitem{PanSil} \newblock D. Panazzolo and  P.R. da Silva,
\newblock Regularization of discontinuous foliations: Blowing up and sliding conditions via Fenichel Theory,
\newblock \emph{J. Differential Equations} \textbf{263} (2017), 8362--8390.

\bibitem{PerRS} \newblock O.H. Perez, G. Rond\'on and  P.R. da Silva,
\newblock Slow-Fast Normal Forms Arising from Piecewise Smooth Vector Fields,
\newblock \emph{J. Dyn. Control Syst.}, \textbf{29} (2023), 1709--1726.


\bibitem{ST96}
     \newblock J. Sotomayor and M. A. Teixeira,
     \newblock Regularization of discontinuous vector fields,
     \newblock \emph{In Proceedings of the International Conference on Differential Equations}, Lisboa, (1996), 207--223.
\end{thebibliography}
\end{document}